\documentclass[11pt,a4paper]{amsart}
\usepackage[utf8]{inputenc}
\usepackage[T1]{fontenc}
\usepackage{amsmath,amssymb,amsthm,mathtools}
\usepackage[margin=1.05in]{geometry}
\usepackage{hyperref}
\usepackage{enumitem}
\usepackage{booktabs,array,tabularx}
\usepackage{tikz-cd}
\usepackage{tikz}
\usetikzlibrary{arrows.meta,calc,patterns}
\tikzset{
  vt/.style={circle,fill=black,inner sep=1.5pt},
  vo/.style={circle,draw=black!30,fill=white,inner sep=1.5pt},
  eon/.style={thick,black},
  eoff/.style={thick,black!25,dash pattern=on 2pt off 2pt},
  bar/.style={line width=2pt,black!80,-{Stealth[length=3pt]}},
  barfin/.style={line width=2pt,black!55},
  lab/.style={font=\scriptsize},
}
\newcommand{\oldpath}[5]{\begin{tikzpicture}[baseline=(c.base),x=0.62cm]
 \coordinate (c) at (0,0);
 \draw[#4] (0,0)--(2,0); \draw[#5] (2,0)--(3,0);
 \node[#1,label={[lab]below:$v_1$}] at (0,0) {};
 \node[#2,label={[lab]below:$v_2$}] at (2,0) {};
 \node[#3,label={[lab]below:$v_3$}] at (3,0) {};\end{tikzpicture}}
\newcommand{\newpath}[7]{\begin{tikzpicture}[baseline=(c.base),x=0.62cm]
 \coordinate (c) at (0,0);
 \draw[#5] (0,0)--(1,0); \draw[#6] (1,0)--(2,0); \draw[#7] (2,0)--(3,0);
 \node[#1,label={[lab]below:$v_1$}] at (0,0) {};
 \node[#2,label={[lab]below:$v_F$}] at (1,0) {};
 \node[#3,label={[lab]below:$v_2$}] at (2,0) {};
 \node[#4,label={[lab]below:$v_3$}] at (3,0) {};\end{tikzpicture}}

\theoremstyle{plain}
\newtheorem{theorem}{Theorem}[section]
\newtheorem{proposition}[theorem]{Proposition}
\newtheorem{lemma}[theorem]{Lemma}
\newtheorem{corollary}[theorem]{Corollary}
\newtheorem{question}[theorem]{Question}
\theoremstyle{definition}
\newtheorem{definition}[theorem]{Definition}
\newtheorem{convention}[theorem]{Convention}
\newtheorem{example}[theorem]{Example}
\newtheorem{hypothesis}[theorem]{Hypothesis}
\theoremstyle{remark}
\newtheorem{remark}[theorem]{Remark}

\newcommand{\Xt}{\widetilde X}
\newcommand{\I}{\mathcal I}
\newcommand{\J}{\mathcal J}
\newcommand{\Hn}{\mathcal H^{2n}}
\newcommand{\lct}{\operatorname{lct}}
\newcommand{\ord}{\operatorname{ord}}
\newcommand{\Jac}{\operatorname{Jac}}
\newcommand{\Exc}{\operatorname{Exc}}
\newcommand{\Supp}{\operatorname{Supp}}
\newcommand{\codim}{\operatorname{codim}}
\newcommand{\DPH}{\mathrm{DPH}}

\newcommand{\R}{\mathbb R}
\newcommand{\C}{\mathbb C}
\newcommand{\Q}{\mathbb Q}
\newcommand{\Z}{\mathbb Z}
\newcommand{\LOG}{|\!\log\varepsilon|}

\title[Divisorial persistent homology and volume asymptotics]{Divisorial persistent homology\\ and volume asymptotics of complex pairs}
\dedicatory{Dedicated to the memory of Maria Aparecida Ruas (Cidinha)}
\author{Nivaldo Grulha}
\address{Instituto de Ci\^encias Matem\'aticas e de Computa\c c\~ao, Universidade de S\~ao Paulo, S\~ao Carlos, SP, Brazil.}
\email{njunior@icmc.usp.br}
\thanks{The author acknowledges the support of FAPESP, Brazil, grant no.\ 2019/21181-0, and CAPES, Brazil, MATH-AmSud program, grant no.\ 88881.179491/2025-01.}

\begin{document}

\begin{abstract}
To a log-resolution $\pi\colon\Xt\to X$ of a complex pair $(X,\I)$ we associate the weights
\[
\lambda_E=\frac{a_E+1}{\nu_E}
\]
of the prime divisors $E$ of the total transform of $\I$, where $\nu_E$ is the vanishing order of $\I$ and $a_E$ the Jacobian order of $\pi$ along $E$ (for singular $X$, the Mather discrepancy). On Jacobian-adapted resolutions, those on which the Jacobian ideal is invertible, we prove that $\min_E\lambda_E$ is the volume-growth exponent of the energy $K_\I=\sum_j|f_j|^2$ with respect to the Hausdorff measure of real dimension $2n$ on $X$: it is the local form of the log canonical threshold for smooth $X$, and of the Mather log canonical threshold of Ishii for singular $X$; a resonance multiplicity gives the logarithmic correction. For algebraic $X$ the weight is $\operatorname{codim}C_X(v)/v(\mathfrak a)$ in the arc space. The main object of the paper is the divisorial persistent homology $\DPH$: the persistent homology of the dual complex of the resolution, filtered by the full subcomplexes on the vertices of weight at most $\alpha$. It records the incidences of the divisors, which the threshold ignores.

We prove that $\DPH$ is invariant under blow-ups along smooth centres in normal crossing with the boundary, provided the resolution is Jacobian-adapted; the $A_2$ singularity shows that adaptedness cannot be dropped. Independence of the adapted resolution holds unconditionally for complex algebraic pairs, by the weak factorization theorem, and for normal surface singularities; for non-algebraic complex analytic pairs of dimension $\ge3$ it remains a hypothesis. For an algebraic pair with $V(\I)$ complete we identify $\DPH$, with rational coefficients, with the weight-zero part of the cohomology of the union of the divisors of weight at most $\alpha$, through Deligne's mixed Hodge theory. For the rational double points $A_n$ the minimal resolution is adapted exactly when $n$ is odd, and $\DPH_0$ has a single finite bar of length $(n-1)/2$: within this family it determines $n$, which the threshold and the multiplicity do not; this is detection inside a family, not a classification of germs. We compare the construction with the weighted skeleta of Nicaise--Xu and with the unweighted dual-complex analysis of de Fernex--Koll\'ar--Xu. The real analytic case, where the weights are halved, is treated in a companion paper.
\end{abstract}

\maketitle

\section*{Introduction}

A resolution of a complex pair $(X,\I)$ supplies numerical data and incidence data. The numerical data are the divisorial ratios $\lambda_E$ defined below: their minimum controls the leading volume exponent of the energy
\[
K_\I=\sum_j|f_j|^2 ,
\]
which for smooth $X$ is the log canonical threshold of $\I$ and for singular $X$ the Mather log canonical threshold of \cite{Ish13}. The minimum alone does not record how the corresponding divisors intersect. This paper studies the persistent homology of the dual complex with vertices weighted by these ratios. It is proposed as a refinement of the pair $(\lambda_{\min},m)$ of leading exponent and logarithmic order, not as an asymptotic expansion of the volume: only the first weight is tied to the volume (Theorem \ref{thm:volume}), and the analytic meaning of the later bars is left open.

Let $\pi\colon\Xt\to X$ be a log-resolution of a complex pair $(X,\I)$. To each prime divisor $E$ of the boundary $D_\pi$ we associate the weight
\[
\lambda_E=\frac{a_E+1}{\nu_E},
\]
where $\nu_E=\ord_E(\I)$ and $a_E$ is the order along $E$ of the Jacobian ideal of $\pi$; $\lambda_E=+\infty$ if $\nu_E=0$. For singular $X$, $a_E$ is the Mather discrepancy, not the ordinary discrepancy and not the Mather--Jacobian one. Thus $\lct(\I)=\min_E\lambda_E$, and the exponent of the volume of the energy sublevel sets is $\min_E\lambda_E$. The dual complex $\Delta(\pi)$ has a vertex for each prime divisor of the boundary $D_\pi$ and a simplex for each connected component of an intersection of divisors (Definition~\ref{def:dual}), and $\Delta(\pi)^{\le\alpha}$ is its full subcomplex on the vertices of weight at most $\alpha$. The principal object is
\[
\DPH^\alpha_k(X,\I;\pi)=H_k\bigl(\Delta(\pi)^{\le\alpha};\Z\bigr),
\]
the divisorial persistent homology of the pair, for $\pi$ Jacobian-adapted when $X$ is singular. Unlike the volume exponent, it records the incidence relations among the divisors.

\medskip\noindent\textbf{Main results.}
\begin{enumerate}[label=\textup{(\arabic*)},leftmargin=2em]
\item \emph{Volume asymptotics} (Theorem~\ref{thm:volume}). On a Jacobian-adapted log-resolution, for every sufficiently small compact neighbourhood $B$ of a point $p\in V(\I)$,
\[
\Hn\bigl(B\cap\{K_\I\le\varepsilon\}\bigr)\asymp\varepsilon^{\lambda_{\min}(p)}\LOG^{m(p)-1},
\]
where $\Hn$ is the Hausdorff measure of real dimension $2n$ on $X$, $\lambda_{\min}(p)=\min\{\lambda_E:p\in\pi(E)\}$ is independent of $\pi$ (Theorem~\ref{thm:divcomp}), and $m(p)$ is the largest number of components of weight $\lambda_{\min}(p)$ through a common point of $\pi^{-1}(p)$. The exponent and the logarithmic order are read off from the lowest sublevel complex of $\Delta(\pi)$ over $p$ (Proposition~\ref{prop:mp}). For algebraic $X$ the weight is the arc-space quantity $\operatorname{codim}C_X(v)/v(\mathfrak a)$ of \cite{Ish13} (Proposition~\ref{prop:arc}).
\item \emph{Invariance under one blow-up} (Theorem~\ref{thm:blowup-inv}). If $\pi$ is Jacobian-adapted and $\rho$ is a blow-up along a smooth centre in normal crossing with the boundary, then $\Delta(\pi)^{\le\alpha}\simeq\Delta(\pi\circ\rho)^{\le\alpha}$, compatibly with the inclusions. The numerical ingredient is a lower bound for the weight of the new divisor, together with an upper bound when the centre is a boundary stratum (Lemma~\ref{lem:mediant}). Without adaptedness the invariance fails (Example~\ref{ex:A2fail}).
\item \emph{Independence of the resolution} (Theorem~\ref{thm:invariance}). Under a weak factorization hypothesis \textup{(WF)} (any two adapted log-resolutions are connected by snc blow-ups and their inverses through adapted ones), $\DPH$ does not depend on the adapted resolution. The hypothesis holds for complex algebraic pairs, by weak factorization \cite{AKMW02} (Corollary~\ref{cor:WFalg}), and for normal surface germs with $\mathfrak m_0$-primary ideal (Theorem~\ref{thm:surface}); for non-algebraic complex analytic pairs of dimension $\ge3$ it is a hypothesis.
\item \emph{Hodge-theoretic meaning} (Proposition~\ref{prop:hodge}). For an algebraic pair with $V(\I)$ complete, $H^k(\Delta(\pi)^{\le\alpha};\Q)\cong\operatorname{gr}^W_0H^k(D^{\le\alpha};\Q)$, where $D^{\le\alpha}$ is the union of the divisors of weight at most $\alpha$, compatibly with the inclusions: with rational coefficients, $\DPH$ is the weight-zero part of a filtered mixed Hodge structure. The higher weights are not controlled by $\DPH$ (Remark~\ref{rem:excnotinv}).
\item \emph{The $A_n$ family} (Example~\ref{ex:Anexplicit}). For the rational double point $A_n$ the minimal resolution is adapted exactly when $n$ is odd, and the module computed on an adapted resolution recovers $n$ within the family, which $\lct$ and $m(0)$ alone do not. This shows that $\DPH$ \emph{detects} finer information than $(\lambda_{\min},m)$; it is not claimed to \emph{classify} germs (Remark~\ref{rem:detection}).
\end{enumerate}
The volume asymptotics sees only the first weight over each point; the later critical values of $\DPH$ carry finer divisorial and topological information, whose analytic meaning is left open.

\subsection*{Scope of the results}\label{subsec:scope}
Table~\ref{tab:scope} states which setting each result is proved in. The three columns are: complex algebraic pairs; germs of normal complex analytic surfaces with $\mathfrak m_0$-primary ideal; and complex analytic pairs in any dimension, where the independence of the resolution holds only under \textup{(WF)}. Throughout, ``adapted'' means Jacobian-adapted, a condition that is empty when $X$ is smooth.

\begin{table}[ht]
\centering\footnotesize\renewcommand{\arraystretch}{1.35}
\begin{tabularx}{\textwidth}{@{}>{\raggedright\arraybackslash}p{3.3cm}>{\raggedright\arraybackslash}X>{\raggedright\arraybackslash}X>{\raggedright\arraybackslash}X@{}}
\toprule
Result & Complex algebraic & Normal surface germs & Complex analytic, any $n$\\
\midrule
Domination of log-resolutions (Prop.~\ref{prop:domination}) & \textbf{P} (Steps~1--2; resolution of indeterminacy) & \textbf{P} (elementary, Lemma~\ref{lem:surfdom}) & \textbf{P} (Hironaka flattening)\\
Arc-space formula (Prop.~\ref{prop:arc}) & \textbf{P}, from \cite{Ish13} & --- & ---\\
Local threshold $\lambda_{\min}$ intrinsic (Thm.~\ref{thm:divcomp}) & \textbf{P}, on any log-resolution & \textbf{P} & \textbf{P}\\
Volume asymptotics and $m(p)$ (Thm.~\ref{thm:volume}, Prop.~\ref{prop:mp}) & \textbf{P}$_{\rm a}$ & \textbf{P}$_{\rm a}$ & \textbf{P}$_{\rm a}$\\
Invariance under one snc blow-up (Thm.~\ref{thm:blowup-inv}) & \textbf{P}$_{\rm a}$ & \textbf{P}$_{\rm a}$ & \textbf{P}$_{\rm a}$\\
Independence of the resolution (Thm.~\ref{thm:invariance}) & \textbf{P}$_{\rm a}$ (Prop.~\ref{prop:WFalg}, Cor.~\ref{cor:WFalg}; uses \cite{AKMW02}) & \textbf{P}$_{\rm a}$, unconditional (Thm.~\ref{thm:surface}) & \textbf{C} (Rem.~\ref{rem:WFreal})\\
Weight-zero Hodge theory (Prop.~\ref{prop:hodge}) & \textbf{P}$_{\rm a}$, $V(\I)$ complete & \textbf{P}$_{\rm a}$ (the divisors are curves) & ---\\
The family $A_n$ (Ex.~\ref{ex:Anexplicit}) & \textbf{P}$_{\rm a}$ & \textbf{P}$_{\rm a}$ & ---\\
Analytic meaning of later bars & open (Questions~\ref{q:poles}--\ref{q:jump}) & open & open\\
\bottomrule
\end{tabularx}
\caption{\textbf{P}: proved in the setting of the column. \textbf{P}$_{\rm a}$: proved, on Jacobian-adapted log-resolutions when $X$ is singular. \textbf{C}: conditional on the weak factorization hypothesis (WF) (Hypothesis~\ref{hyp:WF}). ``---'': not applicable. The algebraic column also needs the compactification and boundary-adjustment steps of Proposition~\ref{prop:WFalg}.}
\label{tab:scope}
\end{table}

The invariants compared in these results form a hierarchy, made explicit in Section~\ref{subsec:hierarchy}: the pair $(\lambda_{\min},m)$ and the module $\DPH$ are both read from the weighted dual complex, but neither determines the other, and the volume asymptotics sees only the former.

\subsection*{Relation to previous work}\label{sec:previous-work}

\paragraph{Weights and log canonicity.}
Let $X$ be a smooth complex variety and $\pi\colon\Xt\to X$ a log-resolution, with $\I\mathcal O_{\Xt}=\mathcal O_{\Xt}(-\sum_E\nu_EE)$ and $K_{\Xt/X}=\sum_E(A_E-1)E$. For $c>0$ the divisor $\Delta_c=c\sum_E\nu_EE-K_{\Xt/X}$ has coefficient $c\nu_E-(A_E-1)$ along $E$, which, for $\nu_E>0$, is at least $1$ exactly when $c\ge\lambda_E=A_E/\nu_E$. Hence $\Delta(\pi)^{\le\alpha}$ is the full subcomplex on the divisors with coefficient at least $1$ in $\Delta_\alpha$, that is, on the divisors computing the non-klt locus of $(X,\alpha\,\I)$. This is a reading on a smooth snc model, and it differs in general from the complex $\mathcal D(\Delta^{=1})$ of a dlt modification used in \cite{dFKX17}, because $\Delta_c$ may have coefficients exceeding $1$. For singular $X$ the numerator $A_E$ is the Mather log discrepancy, and the dictionary is not applied.

\paragraph{Nicaise--Xu.}
For $X$ smooth and $\I=(f)$, the weight $\operatorname{wt}(E)=A_E/N_E$ of \cite{NX16} is our $\lambda_E$. Their Theorem~2.4 states that if $\operatorname{wt}$ takes a common value $w$ on the divisors $E_j$, $j\in J$, of a face whose stratum $C$ meets $h^{-1}(x)$, with no other $E_i$ meeting $C\cap h^{-1}(x)$, then $w=\lct_x$; Theorem~3.5 uses it for Veys' conjecture on poles of maximal order. Proposition~\ref{prop:mp} describes $m(p)-1$ through the same faces, on a Jacobian-adapted log-resolution; the volume estimate itself is proved independently (Theorem~\ref{thm:volume}).

Their Theorem~4.10 assumes $\operatorname{div}f$ reduced at $x$ and a projective log-resolution that is an isomorphism outside $\operatorname{div}f$. If the pair is log canonical at $x$, $\mathrm{Sk}(Y,x)$ collapses onto the essential skeleton $\mathrm{Sk}(f,x)$, simultaneously for every sublevel complex $\mathrm{Sk}(Y,x)^{\le w}$, $w\ge1$. If it is not, the same holds for the exceptional subcomplexes and $w\ge\lct_x$. So \cite{NX16} already contains the weighted vertex filtration and these collapses in the smooth principal setting; the collapse is local at $x$. Our results instead compare the sublevel complexes of different resolutions, globally, and allow singular $X$ (with the Mather discrepancy) and non-principal ideals. Their Examples~4.15 and~4.16 show that the log-canonical hypothesis cannot be dropped. In Example~4.15 ($X=\mathbb A^2$, $f$ the product of an irreducible nodal curve and a generic line, $\pi$ the blow-up of the origin) the exceptional divisor has $\lambda=2/3$ and the strict transforms have $\lambda=1$; the curve meets the exceptional divisor in two points and the line in one, so the dual complex is a circle (two components of an intersection give two edges), $\DPH_1$ has the bar $[1,\infty)$, and the essential skeleton is a point. The $A_n$ computation of Example~\ref{ex:An} lies outside the smooth setting of \cite{NX16}.

\paragraph{de Fernex--Koll\'ar--Xu.}
For the unweighted dual complex, \cite[Paragraph~9]{dFKX17} shows that blowing up a stratum of an snc pair gives a stellar subdivision and blowing up a centre in the boundary that is not a stratum attaches a cone over a join; the complexes are homotopy equivalent, not necessarily PL homeomorphic. Their Proposition~11 deduces PL invariance from weak factorization \cite{AKMW02}. Lemma~\ref{lem:dualblowup} is the filtered analogue, and the new input is Lemma~\ref{lem:mediant}. Its lower bound $\lambda_F\ge\min_{j\in\sigma,\nu_j>0}\lambda_j$ holds on every log-resolution and ensures that the new vertex attaches along a cone. Its upper bound $\lambda_F\le\max_{j\in\sigma}\lambda_j$ needs $\pi$ Jacobian-adapted, the centre a component of $\bigcap_{j\in\sigma}E_j$ ($r=s$) and all $\nu_j>0$; it prevents a simplex from leaving a sublevel complex before its subdivision enters, and it fails otherwise (Example~\ref{ex:A2fail}).

\paragraph{Valuation spaces and normalized volumes.}
For $X$ smooth and complex algebraic, $\lambda$ is the restriction of $A(v)/v(\I)$ to divisorial valuations \cite{JM12}, and the sublevel complexes are skeleta intersected with sublevel sets of a weight function \cite{MN15}; the literal sublevel set of the affine weight on a simplicial realization retracts onto the full subcomplex on the low vertices, which is the object used here (Remark~\ref{rem:literature}). For singular $X$ the numerator is the Mather log discrepancy and the weight is the arc-space function of Proposition~\ref{prop:arc}, not that of \cite{JM12}. The normalized volume $\widehat{\mathrm{vol}}(v)=A(v)^n\mathrm{vol}(v)$ of \cite{Li18,LLX20,XZ21} uses multiplicities of valuation ideals, not the Hausdorff measure of the sets $\{K_\I\le\varepsilon\}$; the two are linked through log canonical thresholds, for a klt singularity by $\widehat{\mathrm{vol}}(x,X)=\inf_{\mathfrak a_\bullet}\lct(\mathfrak a_\bullet)^n\mathrm{mult}(\mathfrak a_\bullet)$ \cite[(2.1)]{XZ21}.

\paragraph{Weak factorization.}
Hypothesis~\ref{hyp:WF} is a factorization statement relative to $X$ and the boundary. For complex algebraic pairs it follows from \cite[Thm.~0.3.1]{AKMW02} after compactification and putting the two boundaries in common position, as in \cite[Prop.~11]{dFKX17}; adaptedness of the intermediate models is automatic (Lemma~\ref{lem:adaptdom}(i)). This is Proposition~\ref{prop:WFalg}. For non-algebraic complex analytic pairs \textup{(WF)} remains a hypothesis (Remark~\ref{rem:WFreal}): \cite{AKMW02} is a theorem on algebraic varieties, and we use no analytic version of it.

\subsection*{Organization}
Section \ref{sec:pairs} introduces complex pairs, log-resolutions and divisorial valuations. Section \ref{sec:volume} develops the volume asymptotics and the normalization of the weights $\lambda_E$. Section \ref{sec:threshold} proves that the local threshold is computed on any log-resolution. Section \ref{sec:dual} introduces the filtered dual complex, proves invariance under snc blow-ups and discusses independence of the resolution and its Hodge-theoretic meaning. Section \ref{sec:surfaces} treats normal surface singularities and the $A_n$ family. Section \ref{sec:outlook} concludes.

\section{Analytic pairs, log-resolutions and divisorial data}\label{sec:pairs}

\begin{convention}\label{conv:setting}
Throughout, $X$ is a reduced complex analytic space of pure dimension $n$ (a representative of a germ, a compact space, or an algebraic variety; nothing below depends on this choice), $\I\subset\mathcal O_X$ is a coherent ideal sheaf with generators $f_1,\dots,f_r$ near any point, and
\[
K_\I=\sum_{j=1}^r|f_j|^2
\]
is the associated energy. We assume that $\I$ is generically non-zero on every irreducible component of $X$ (otherwise $K_\I\equiv0$ on a component of maximal dimension, the sublevel sets do not shrink and there is nothing to measure); its asymptotic equivalence class (Definition \ref{def:asymp}) depends only on $\I$ (Proposition \ref{prop:asymp}). We write $U_\varepsilon(\I)=\{K_\I\le\varepsilon\}$ (Figure~\ref{fig:Ueps}); these sets shrink to the zero locus $V(\I)$ as $\varepsilon\to0$, and the question of Section \ref{sec:volume} is how fast their volume, for the Hausdorff measure $\Hn$ of real dimension $2n$ of the underlying real analytic space, goes to zero. Real analytic pairs, where the weights are halved, are treated in the companion paper \cite{GruR}.
\end{convention}

\begin{figure}[ht]
\centering
\begin{tikzpicture}[x=1.9cm,y=1.9cm]
 \begin{scope}
  \draw[->,black!45] (-1.15,0)--(1.15,0) node[right,lab]{$\Re x$};
  \draw[->,black!45] (0,-1.15)--(0,1.15) node[above,lab]{$\Im x$};
  \foreach \e/\g in {0.5/12,0.1/28,0.01/50}{
    \pgfmathsetmacro{\r}{sqrt(\e)}
    \draw[fill=black!\g,draw=black!70,thick] (0,0) circle (\r);}
  \fill (0,0) circle (1.1pt);
  \node[lab] at (0.62,0.82) {$U_\varepsilon(\mathcal I)$};
  \node[lab,below] at (0,-1.18) {$\mathcal I=(x)$, $K_{\mathcal I}=|x|^2$};
 \end{scope}
 \begin{scope}[xshift=3.4cm]
  \draw[->,black!45] (-1.15,0)--(1.15,0) node[right,lab]{$\Re x$};
  \draw[->,black!45] (0,-1.15)--(0,1.15) node[above,lab]{$\Im x$};
  \foreach \e/\g in {0.5/12,0.1/28,0.01/50}{
    \pgfmathsetmacro{\r}{\e^(1/6)}
    \draw[fill=black!\g,draw=black!70,thick] (0,0) circle (\r*0.62);}
  \fill (0,0) circle (1.1pt);
  \node[lab] at (0.62,0.82) {$U_\varepsilon$};
  \node[lab,below] at (0,-1.18) {$\mathcal I=(x^3)$, $K_{\mathcal I}=|x|^6$};
 \end{scope}
\end{tikzpicture}
\caption{The sublevel sets $U_\varepsilon(\mathcal I)=\{K_{\mathcal I}\le\varepsilon\}$ for decreasing $\varepsilon$ (darker = smaller $\varepsilon$), in the only drawable case, $X=\mathbb C$. Left: discs of radius $\sqrt\varepsilon$, of area $\pi\varepsilon$ ($\lambda=1$). Right: discs of radius $\varepsilon^{1/6}$, of area $\pi\varepsilon^{1/3}$ ($\lambda=1/3$). In $\mathbb C^n$ the sets live in $\mathbb R^{2n}$; the local model of Section~\ref{sec:volume} is a product of such discs.}\label{fig:Ueps}
\end{figure}

\begin{convention}[Scope in the complex category]\label{conv:scope}
All complex analytic spaces are reduced and of pure dimension. Unless a statement concerns a compact space or an algebraic variety, it is understood on sufficiently small representatives of the germs and morphisms involved, which may be shrunk without further mention; properness refers to the chosen representatives. The volume estimates (Section \ref{sec:volume}) are local near the point considered; the local threshold (Section \ref{sec:threshold}) is defined on the chosen representative; invariance under one snc blow-up of an adapted log-resolution (Theorem \ref{thm:blowup-inv}) is a local computation on complex manifolds. The comparison of arbitrary adapted log-resolutions is conditional on Hypothesis \ref{hyp:WF}, except for complex algebraic pairs (Corollary \ref{cor:WFalg}) and normal surface germs (Theorem \ref{thm:surface}). No complexification is needed in this paper; weak factorization is invoked only for algebraic pairs.
\end{convention}

\begin{definition}\label{def:asymp}
Two non-negative functions $K_1,K_2$ are \emph{asymptotically equivalent}, $K_1\asymp K_2$, if $cK_1\le K_2\le CK_1$ near their common zero locus for some $c,C>0$.
\end{definition}

\begin{proposition}\label{prop:asymp}
If $f$ and $g$ are two generating systems of $\I$ near a point, then $K_f\asymp K_g$. Consequently $U_{c\varepsilon}(K_f)\subset U_\varepsilon(K_g)\subset U_{C\varepsilon}(K_f)$ for small $\varepsilon$.
\end{proposition}
\begin{proof}
Write $g=Af$, $f=Bg$ with analytic matrices $A,B$ bounded near the point; then $\|g\|\le\|A\|\,\|f\|$ and $\|f\|\le\|B\|\,\|g\|$, and squaring gives the claim.
\end{proof}

A \emph{log-resolution} of $(X,\I)$ is a proper birational morphism $\pi\colon\Xt\to X$ from a complex manifold such that $\I\cdot\mathcal O_{\Xt}=\mathcal O_{\Xt}(-\sum_E\nu_EE)$ is principal and
\[
D_\pi:=\Exc(\pi)\cup\Supp(\pi^{-1}\I)
\]
is a simple normal crossings divisor; we call $D_\pi$ the \emph{boundary} of the pair $(\Xt,D_\pi)$, in the sense of log pairs in birational geometry. Log-resolutions exist, and canonical ones can be chosen, by Bierstone--Milman \cite{BM97} (extending Hironaka's original algebraic existence theorem \cite{Hir64} to the complex analytic category). The boundary $D_\pi$ may contain non-exceptional components (strict transforms of components of $V(\I)$ of codimension one) and exceptional components with $\nu_E=0$ (divisors over $\operatorname{Sing}X$ outside $V(\I)$); both are kept.

\subsection{Blow-ups along snc centres}
\begin{definition}\label{def:sncblowup}
Let $\pi\colon Y\to X$ be a log-resolution. An \emph{snc blow-up} of $Y$ is the blow-up $\rho\colon Y'\to Y$ along a smooth closed irreducible subset $Z\subset Y$ of codimension at least two having simple normal crossings with $D_\pi$ (an \emph{snc centre}); its exceptional divisor is $F$. Then $\I\cdot\mathcal O_{Y'}=\rho^*(\I\cdot\mathcal O_Y)$ is principal and $D_{\pi\circ\rho}=\rho^{-1}(D_\pi)\cup F$ (as reduced supports) is simple normal crossings, so $\pi\circ\rho$ is again a log-resolution.
\end{definition}

\begin{figure}[ht]
\centering
\begin{tikzpicture}[x=1cm,y=1cm,
  cur/.style={very thick,black},
  ex/.style={thick,black!60},
  tag/.style={font=\scriptsize,fill=white,inner sep=1pt}]
 \begin{scope}
  \draw[cur,domain=-1.1:1.1,samples=50,smooth,variable=\t] plot ({\t*\t},{\t*\t*\t});
  \node[tag] at (1.45,1.2) {$C$};
  \node[lab] at (0.2,-1.55) {$X=\C^2$, $\mathcal I=(y^2-x^3)$};
 \end{scope}
 \draw[->,thick] (1.9,0)--(2.7,0) node[midway,above,lab]{$\rho_1$};
 \begin{scope}[xshift=4.3cm]
  \draw[ex] (-1.3,0)--(1.3,0); \node[tag] at (1.1,0.28) {$E_1$};
  \draw[cur,domain=-0.9:0.9,samples=40,smooth,variable=\t] plot ({\t},{\t*\t*0.9+0});
  \node[tag] at (1.0,0.75) {$C'$};
  \node[lab] at (0,-1.55) {$\nu_{E_1}=2$};
 \end{scope}
 \draw[->,thick] (6.2,0)--(7.0,0) node[midway,above,lab]{$\rho_2$};
 \begin{scope}[xshift=8.6cm]
  \draw[ex] (-1.3,0)--(1.3,0); \node[tag] at (1.1,-0.25) {$E_1'$};
  \draw[ex] (0,-1.1)--(0,1.2); \node[tag] at (0.3,1.05) {$E_2$};
  \draw[cur] (-0.95,-0.95)--(1.05,1.05); \node[tag] at (1.2,1.2) {$C''$};
  \fill (0,0) circle (1.5pt);
  \node[lab] at (0,-1.55) {$\nu_{E_2}=3$, triple point (not snc)};
 \end{scope}
 \draw[->,thick] (10.3,0)--(11.1,0) node[midway,above,lab]{$\rho_3$};
 \begin{scope}[xshift=12.7cm]
  \draw[ex] (-1.1,-0.9)--(-0.35,0.95); \node[tag] at (-0.9,-0.9) {$E_1''$};
  \draw[ex] (1.1,-0.9)--(0.35,0.95); \node[tag] at (0.95,-0.9) {$E_2'$};
  \draw[ex] (-1.0,0.05)--(1.0,0.05); \node[tag] at (1.15,0.05) {$E_3$};
  \draw[cur] (0,-0.9)--(0,0.95); \node[tag] at (0.0,1.15) {$C'''$};
  \fill (-0.67,0.05) circle (1.5pt); \fill (0,0.05) circle (1.5pt); \fill (0.67,0.05) circle (1.5pt);
  \node[lab] at (0,-1.55) {$\nu_{E_3}=6$, snc};
 \end{scope}
\end{tikzpicture}
\caption{A log-resolution of the cusp $\mathcal I=(y^2-x^3)$ by three successive point blow-ups; $\nu_{E_i}=\ord_{E_i}\mathcal I$ is $2,3,6$ (primes denote strict transforms). After $\rho_1$ the strict transform $C'$ is smooth but tangent to $E_1$; after $\rho_2$ the three curves $E_1',E_2,C''$ pass through one point (not snc); after $\rho_3$ the boundary $D_\pi=E_1''\cup E_2'\cup E_3\cup C'''$ is simple normal crossings, $E_3$ meeting the other three transversally in three distinct points. These three steps \emph{construct} a log-resolution; they are not snc blow-ups in the sense of Definition~\ref{def:sncblowup}, which starts from a log-resolution. Further point blow-ups, for instance at $E_3\cap C'''$, are snc blow-ups and give log-resolutions dominating it. The pictures are real sketches of complex curves and are schematic.}
\label{fig:cusp}
\end{figure}

The figure above illustrates the \emph{construction} of a log-resolution, not snc blow-ups of an already resolved pair: once the final snc boundary is reached, further point blow-ups are snc blow-ups in the sense of Definition~\ref{def:sncblowup}.

The birational input is the following comparison of two log-resolutions; it follows from Hironaka's flattening theorem and canonical principalization.

\begin{proposition}\label{prop:domination}
Any two log-resolutions $Y,Y'$ of $(X,\I)$ are dominated by a log-resolution $W$ obtained from $Y$ by a finite sequence of snc blow-ups, together with a proper birational morphism $W\to Y'$ over $X$.
\end{proposition}
\begin{proof}
\emph{Step 1.} Let $\Gamma\subset Y\times_X Y'$ be the closure of the graph of the bimeromorphic map $Y\dashrightarrow Y'$; the projection $p\colon\Gamma\to Y$ is proper and bimeromorphic. By Hironaka's flattening theorem \cite[Thm.~4.4 and Cor.~2 (Chow's lemma)]{Hir75} there is a proper (i.e.\ locally finite) sequence of blow-ups $\beta\colon Y_1\to Y$ along nowhere dense closed analytic centres $Z_1,Z_2,\dots$ (finitely many over any relatively compact open set, hence finitely many on the representatives considered) (not necessarily smooth; smoothness is not used below) such that the strict transform of $\Gamma$ is flat over $Y_1$. Here $Y_1$ is reduced and irreducible (a blow-up of the connected manifold $Y$ along nowhere dense centres; we work on each connected component of $Y$ separately), but need not be smooth or normal, and no such property is used. Let $\Gamma_1\to Y_1$ be the strict transform of $\Gamma$, a proper flat bimeromorphic morphism of irreducible reduced spaces of the same dimension. By the dimension formula for flat morphisms its fibres have dimension $0$, so it is finite; a finite flat morphism is locally free, and its rank is $1$ because it is bimeromorphic and $Y_1$ is irreducible, so $\mathcal O_{Y_1}\to(\Gamma_1\to Y_1)_*\mathcal O_{\Gamma_1}$ is an isomorphism and $\Gamma_1\to Y_1$ is an isomorphism. Hence $Y_1\to Y$ factors through $\Gamma$, and we get $Y_1\to Y'$ over $X$. If $X,Y,Y'$ are algebraic, one may instead resolve the indeterminacy of the birational map $Y\dashrightarrow Y'$ by Hironaka's theorem \cite{Hir64}; the centres $Z_i$ are then smooth and algebraic. In either case the centres need not be in normal crossing with the boundary; this is why Step 2 replaces them.

\emph{Step 2.} We replace $\beta$ by a tower $W_k\to\dots\to W_1\to W_0=Y$ of snc blow-ups, and morphisms $W_i\to Y^{(i)}:=\mathrm{Bl}_{Z_i}\cdots\mathrm{Bl}_{Z_1}Y$. Given $W_{i-1}$, let $\mathfrak c_i$ be the inverse image of the ideal of $Z_i$. The canonical principalization of Bierstone--Milman \cite[Thm.~1.10]{BM97} applied to $\mathfrak c_i$ on $(W_{i-1},D)$ is a sequence of blow-ups along smooth centres with simple normal crossings with the boundary, after which $\mathfrak c_i\mathcal O_{W_i}$ is invertible; the universal property of blowing up gives $W_i\to Y^{(i)}$. Put $W=W_k$ and $\nu\colon W\to Y_1\to Y'$, proper and birational. Blow-ups in Step~2 along centres of codimension one are isomorphisms and are omitted, so all centres have codimension at least two as Definition~\ref{def:sncblowup} requires. Finally $W$ is a log-resolution of $(X,\I)$ by Definition~\ref{def:sncblowup}.
\end{proof}

\begin{figure}[ht]
\centering
\begin{tikzcd}[column sep=1.0cm,row sep=0.9cm]
  W=W_k\arrow[r,"\rho_k"]\arrow[d,"\nu"'] & \cdots\arrow[r,"\rho_2"] & W_1\arrow[r,"\rho_1"] & W_0=Y\arrow[d,"\pi"]\\
  Y'\arrow[rrr,"\pi'"'] & & & X
\end{tikzcd}
\caption{Domination of log-resolutions (Proposition~\ref{prop:domination}): $\mu=\rho_1\circ\dots\circ\rho_k$ is a tower of snc blow-ups of $Y$, and $\nu$ is only proper and birational.}
\label{fig:dominated}
\end{figure}

The morphism $W\to Y'$ is in general only proper and birational. This one-sided domination suffices for Sections \ref{sec:threshold}--\ref{sec:surfaces}; a chain of blow-ups and inverse blow-ups between arbitrary adapted log-resolutions is Hypothesis \ref{hyp:WF}. When both log-resolutions are obtained from a common one by sequences of snc blow-ups, or for surface germs (where a proper birational morphism of smooth surface germs is a composition of point blow-ups), Step 1 is not needed.

When $X$ is singular, a second property of $\pi$ enters.

\begin{definition}\label{def:adapted}
Let $\Jac_{\Xt/X}\subset\mathcal O_{\Xt}$ be the Jacobian ideal of $\pi$, generated by the $n\times n$ minors of $D\pi$ in an embedding $X\subset\C^N$ (Remark \ref{rem:mather}). The log-resolution $\pi$ is \emph{Jacobian-adapted} if $\Jac_{\Xt/X}$ is invertible, i.e.\ if $\pi$ factors through the Nash blow-up of $X$ (Remark \ref{rem:mather}); equivalently, if in coordinates adapted to $D_\pi$ (Lemma \ref{lem:adapted}) the ideal is generated by $\prod_j u_j^{a_{E_j}}$ times a unit. When $X$ is smooth, $\Jac_{\Xt/X}=(\det D\pi)$ and every log-resolution is Jacobian-adapted.
\end{definition}

Adaptedness is a condition on the model, not on the pair: it is the algebraic condition that $\pi$ factors through the Nash blow-up, which is the setting in which the Mather discrepancy divisor $\widehat K_{Y/X}$ is defined \cite[Def.~2.1]{Ish13}, and it is exactly what the proofs need (the exact Jacobian density, the upper mediant bound). We do not claim that it is the only natural class of models, and nothing below shows that it is canonical beyond this. In Sections \ref{sec:dual} and \ref{sec:surfaces}, where $X$ may be singular, all log-resolutions are taken Jacobian-adapted.

\begin{lemma}\label{lem:adaptdom}
\begin{enumerate}[label=\textup{(\roman*)},leftmargin=2em]
\item If $\pi\colon Y\to X$ is Jacobian-adapted and $\phi\colon W\to Y$ is a proper birational morphism from a complex manifold, then $\pi\circ\phi$ is Jacobian-adapted, and $\Jac_{W/X}=(\det D\phi)\cdot\phi^*\Jac_{Y/X}$. In particular snc blow-ups preserve Jacobian-adaptedness.
\item Every log-resolution is dominated by a Jacobian-adapted log-resolution obtained from it by a finite sequence of snc blow-ups.
\end{enumerate}
\end{lemma}
\begin{proof}
(i) is the chain rule $\Jac_{W/X}=\Jac_{W/Y}\cdot\phi^*\Jac_{Y/X}$ for the composite $W\to Y\to X\subset\C^N$, together with $\Jac_{W/Y}=(\det D\phi)$, which is principal because $W$ and $Y$ are manifolds of the same dimension; the product of an invertible ideal with a principal one is invertible. (ii) By the canonical principalization algorithm of Bierstone--Milman \cite[Thm.~1.10]{BM97}, applied to $\Jac_{\Xt/X}$ on the pair $(\Xt,D_\pi)$ (the algorithm accepts a simple normal crossings divisor as part of its input, and its centres are normal crossing with it), there is a finite sequence of blow-ups $\phi\colon Y\to\Xt$ along such centres, hence snc blow-ups in the sense of Definition \ref{def:sncblowup}, after which $\phi^*\Jac_{\Xt/X}$ is invertible. By (i), $\Jac_{Y/X}=(\det D\phi)\cdot\phi^*\Jac_{\Xt/X}$ is then a product of two invertible ideals, hence invertible, so $\pi\circ\phi$ is Jacobian-adapted; and $\pi\circ\phi$ is a log-resolution by Definition \ref{def:sncblowup}.
\end{proof}

The class of Jacobian-adapted log-resolutions is therefore cofinal among log-resolutions and closed under domination by smooth models; likewise, by Lemma \ref{lem:adaptdom}, every log-resolution obtained from a Jacobian-adapted one by snc blow-ups is adapted, and every log-resolution is dominated by an adapted one. Moreover, once $Y$ is Jacobian-adapted, $\Jac_{Y/X}$ is invertible near \emph{every} point of $Y$, so the excess $e_Z$ of Lemma \ref{lem:mediant} below vanishes for every snc centre $Z\subset Y$, not only for those actually used to reach $Y$ from a log-resolution: adaptedness, once attained, is never at risk from further snc blow-ups, and the two-sided mediant bound is available at every subsequent step. Everything that needs the \emph{exact} value of the Jacobian order of a divisor created by a blow-up requires it: the volume estimates of Section \ref{sec:volume} (through Lemma \ref{lem:adapted}), the combinatorial form of the resonance multiplicity (Proposition \ref{prop:mp}), the two-sided mediant inequality (Lemma \ref{lem:mediant}), and hence the whole of Sections \ref{sec:dual} and \ref{sec:surfaces}. On a non-adapted resolution the Jacobian order of a new divisor satisfies only a lower bound, and Example \ref{ex:A2fail} shows on $A_2$ that the exact formula, the upper mediant bound, the combinatorial multiplicity and the invariance of the filtered dual complex all fail. The local threshold of Section \ref{sec:threshold} needs only the lower bound and is computed on arbitrary log-resolutions (Remark \ref{rem:singX}).

\begin{definition}\label{def:divdata}
For a prime divisor $E\subset D_\pi$ let $\nu_E=\ord_E(\I)\ge0$ and $a_E=\ord_E(\Jac\pi)\ge0$, the order of vanishing along $E$ of the Jacobian of $\pi$, i.e.\ of the complex $n$-dimensional Jacobian of $\pi$, whose squared modulus is the density of the $2n$-dimensional Hausdorff measure of $X$ (Remark \ref{rem:mather}). Set
\[
A_E=a_E+1,\qquad \lambda_E=\frac{A_E}{\nu_E},
\]
with $\lambda_E=+\infty$ when $\nu_E=0$. The \emph{divisorial spectrum} of $\pi$ is the finite set $\Gamma_\pi(X,\I)=\{\lambda_E:E\subset D_\pi\}\subset(0,+\infty]$.
\end{definition}

\begin{remark}\label{rem:mather}
\emph{The Nash blow-up.} $\operatorname{Nash}(X)\subset X\times\operatorname{Gr}(n,N)$ is the closure of the graph of $x\mapsto T_xX$ over $X_{\mathrm{reg}}$; the projection $\nu\colon\operatorname{Nash}(X)\to X$ is an isomorphism over $X_{\mathrm{reg}}$ and, over $\operatorname{Sing}X$, records the limiting tangent planes \cite{Nob75}. In an embedding $X\subset\C^N$, $\operatorname{Nash}(X)$ is cut out by the $n\times n$ minors of a presentation matrix, so $\pi$ factors as $\Xt\to\operatorname{Nash}(X)\to X$ exactly when the pulled-back ideal of those minors, $\Jac_{\Xt/X}$, is principal; this is the equivalence in Definition \ref{def:adapted}. A log-resolution chosen for an arbitrary ideal $\I$ need not have this property; only one deliberately resolving the Jacobian ideal of $X$ does so automatically \cite{Ish13,dFD14,EI15,EIM16}. We call such a log-resolution \emph{Jacobian-adapted}; the references above speak of resolutions factoring through the Nash blow-up.

For $X$ singular, $a_E$ is the order along $E$ of the ideal $\Jac_{\Xt/X}$ generated by the $n\times n$ minors of $D\pi$ in an embedding $X\subset\C^N$; equivalently, $\Jac_{\Xt/X}\cdot\Omega^n_{\Xt}$ is the image of $\pi^*\Omega^n_X\to\Omega^n_{\Xt}$, which is the description of \cite[\S1]{dFD14}, and $a_E$ is the Mather discrepancy of $\ord_E$, the number $\widehat k_E=\ord_E(\widehat K_{Y/X})$ of \cite[Def.~2.1]{Ish13}. When $X$ is $\Q$-Gorenstein, $\widehat K_{Y/X}\ge K_{Y/X}$ \cite[\S2.2]{Ish13}, so $a_E\ge k_E(X)$. By the area formula, $\pi^*\Hn_X$ is the density $\sum|\text{minors}|^2\,d\mathrm{Leb}$ (Cauchy--Binet for the complex-linear map $D\pi$), and this is $\asymp\prod_E|u_E|^{2a_E}$ in coordinates adapted to $D_\pi$ precisely when $\pi$ is Jacobian-adapted; in general the ideal $\Jac_{\Xt/X}$ is contained in $\mathcal O_{\Xt}(-\sum_Ea_EE)$ and only the upper bound $\lesssim$ holds. Two facts about $a_E$ are used repeatedly.

\emph{The order is non-negative and valuative.} $a_E\ge0$ since $\Jac_{\Xt/X}$ is an ideal of $\mathcal O_{\Xt}$; this differs from the classical discrepancy $k_E(X)$, which can be negative for non-rational normal surface singularities (e.g.\ $k_E=-1$ on the minimal resolution of a minimally elliptic singularity), and coincides with $a_E$ only where $X$ is smooth. And $a_E$ depends only on the valuation $\ord_E$ (Proposition \ref{prop:valinv}).

\emph{The relation with the Jacobian ideal of $X$.} Let $X$ be a normal complex hypersurface $\{h=0\}\subset\C^{n+1}$ (this covers all the examples of Section \ref{sec:surfaces}), $j_X=(\partial h/\partial x_0,\dots,\partial h/\partial x_n)\cdot\mathcal O_X$ its Jacobian ideal and $K_{\Xt/X}$ the relative canonical divisor of a resolution $\pi$. Then
\[
\Jac_{\Xt/X}=j_X\mathcal O_{\Xt}\cdot\mathcal O_{\Xt}(-K_{\Xt/X}),\qquad\text{hence}\qquad a_E=k_E(X)+\ord_E(j_X)
\]
for every divisor $E$ over $X$. Indeed, $\Omega^n_X$ modulo torsion embeds in $\omega_X=j_*\Omega^n_{X_{\mathrm{reg}}}$ (normality), $\omega_X$ is generated by the residue form $\omega_0$ of $dx_0\wedge\dots\wedge dx_n/h$, and $dx_0\wedge\cdots\widehat{dx_i}\cdots\wedge dx_n$ restricts to $\pm(\partial h/\partial x_i)\,\omega_0$; so the image of $\Omega^n_{\C^{n+1}}|_X\to\omega_X$ is $j_X\omega_X$, and $\pi^*\omega_0$ is a meromorphic section of $\omega_{\Xt}$ with divisor $K_{\Xt/X}$, so the image of $\pi^*(j_X\omega_X)$ in the meromorphic $n$-forms is $j_X\mathcal O_{\Xt}\cdot\mathcal O_{\Xt}(-K_{\Xt/X})\otimes\omega_{\Xt}$, which lies in $\omega_{\Xt}$ even when $K_{\Xt/X}$ has negative coefficients; the torsion of $\pi^*\Omega^n_X$ dies in the torsion-free $\Omega^n_{\Xt}$. The same identity holds for normal locally complete intersections, with $j_X$ the ideal of maximal minors of the Jacobian matrix of a set of defining equations, and the resulting formula $a_E=k_E(X)+\ord_E(j_X)$ is \cite[Def.~3.1, Cor.~3.5]{dFD14}. Since $\mathcal O_{\Xt}(-K_{\Xt/X})$ is invertible, the identity gives a criterion: \emph{a resolution of a normal hypersurface (or lci) is Jacobian-adapted iff it principalizes $j_X$.} For rational double points $k_E=0$ on the minimal resolution and $a_E=\ord_E(j_X)$; Example \ref{ex:An} uses the criterion to decide which minimal resolutions of $A_n$ are adapted. This is the convention used in all examples, and $\lct$ below always means the threshold defined by it, the Mather log canonical threshold of \cite{Ish13} (Proposition~\ref{prop:arc}).
\end{remark}

\begin{definition}\label{def:valuation}
A \emph{divisorial valuation over $X$} is a valuation of the form $v=q\,\ord_E$, $q\in\Q_{>0}$, for $E$ a prime divisor of $D_\mu$ on some log-resolution $\mu\colon W\to X$; two such data define the same valuation if they agree as functions on $\mathcal O_X$ near the centre. The \emph{centre} of $v$ on $X$ is $c_X(v)=\mu(E)$, a closed subanalytic subset. For $v=q\ord_E$ set
\[
v(\I)=q\,\nu_E,\qquad A(v)=q\,A_E=q(a_E+1),\qquad \lambda(v)=\frac{A(v)}{v(\I)}\in(0,+\infty].
\]
$A(v)$ is the \emph{log discrepancy} of $v$ (Mather log discrepancy for singular $X$). Both $v(\I)$ and $A(v)$ are homogeneous of degree one in $q$, so $\lambda(qv)=\lambda(v)$; and $v(\I),A(v)$ and $c_X(v)$ depend only on $v$, not on the log-resolution realizing it (Proposition \ref{prop:valinv}).
\end{definition}

The homogeneity is the reason for writing $A(v)$ rather than $a(v)+1$: with the latter, $\lambda$ would depend on the normalization $q$. This is also the notation of \cite{JM12,BFJ08}, where $\lct(\I)=\inf_vA(v)/v(\I)$ for $X$ smooth; our $\lambda$ is that function restricted to divisorial valuations. For singular $X$ our $A(v)$ is the Mather log discrepancy, so $\lambda$ is the corresponding Mather function \cite{Ish13}, which is not the function of \cite{JM12}; for algebraic $X$ it has an arc-space description (Proposition~\ref{prop:arc}).

\begin{proposition}\label{prop:valinv}
Let $\pi$ be a log-resolution.
\begin{enumerate}[label=\textup{(\roman*)},leftmargin=2em]
\item $\Gamma_\pi(X,\I)$ is finite.
\item If $E\subset D_\pi$ and $F\subset D_{\pi'}$, on log-resolutions $\pi,\pi'$, define the same valuation, then $\nu_E=\nu_F$, $a_E=a_F$, and $\pi(E)=\pi'(F)$.
\end{enumerate}
\end{proposition}
\begin{proof}
(i) $D_\pi$ has finitely many components. (ii) Take $W\to Y$ and $\psi\colon W\to Y'$ as in Proposition \ref{prop:domination}, and let $G\subset W$ be the strict transform of $E$. A proper birational morphism between manifolds is an isomorphism over the complement of a closed subset of codimension $\ge2$ of the target (Zariski's main theorem), so $W\to Y$ and $\psi$ are isomorphisms at the generic point of $G$ onto $E$ and onto the prime divisor of $Y'$ with the same valuation, i.e.\ $F$. Hence $\ord_G(\rho^*E)=1$ and $\ord_G(\Jac_{W/Y})=0$, and the chain rule $\Jac_{W/X}=\Jac_{W/Y}\cdot\rho^*\Jac_{Y/X}$ gives $\ord_G(\Jac_{W/X})=a_E$; likewise for $F$ and for the orders of $\I$. Finally $G\to E$ and $G\to F$ are proper and dominant, hence surjective, so $\pi(E)=(\pi\circ\rho)(G)=\pi'(F)$. (For algebraic $X$ the valuative nature of $a_E$ also follows from \cite[Prop.~2.13]{Ish13}, which expresses $\widehat k_E+1$ through the arc space; see Proposition~\ref{prop:arc}.)
\end{proof}

\begin{proposition}[Arc-space description of the weight, after Ishii]\label{prop:arc}
Let $X$ be an irreducible reduced complex algebraic variety of dimension $n$ and $\mathfrak a\subset\mathcal O_X$ a nonzero ideal; so $a_E=\widehat k_E$ is the Mather discrepancy of \cite{Ish13}. For a divisorial valuation $v=q\ord_E$ let $C_X(v)\subset X_\infty$ be the maximal divisorial set of $v$ in the arc space of $X$, the Zariski closure of $\{\alpha\in X_\infty:\ord_\alpha=v\}$ \cite[Def.~2.6]{Ish13}. Then \cite[Prop.~2.13]{Ish13}
\[
\operatorname{codim}\bigl(C_X(v),X_\infty\bigr)=q(\widehat k_E+1)=A(v),
\qquad\text{hence}\qquad
\lambda(v):=\frac{A(v)}{v(\mathfrak a)}=\frac{\operatorname{codim}\,C_X(v)}{v(\mathfrak a)},,
\]
since $v(\mathfrak a)=q\,\nu_E$. In particular $\lambda(v)$ depends only on $v$ and not on the log-resolution, and the local threshold of Definition~\ref{def:localthr} is the local form at $p$ of the Mather log canonical threshold of \cite[Def.~3.3]{Ish13}, $\widehat{\mathrm{lct}}(X,\mathfrak a)=\min_{m\ge1}\operatorname{codim}\bigl(\operatorname{Cont}^{\ge m}(\mathfrak a)\bigr)/m$ \cite[Prop.~3.5]{Ish13}, the minimum over $p$ of the former giving the latter. This is a statement about algebraic varieties; no arc-space analogue is claimed for non-algebraic complex analytic germs, where the weights are identified with the volume exponent through Theorem~\ref{thm:volume}. The proofs of Sections~\ref{sec:volume}--\ref{sec:surfaces} do not use this proposition.
\end{proposition}
\begin{proof}
The first display is \cite[Prop.~2.13]{Ish13} and the second is \cite[Prop.~3.5]{Ish13}, together with the independence of $A(v)$ and $v(\mathfrak a)$ of the log-resolution (Proposition~\ref{prop:valinv}).
\end{proof}

We say that $v$ \emph{reaches} $p\in X$ if $p\in c_X(v)$.

\begin{definition}\label{def:localthr}
For $p\in X$ set
\[
\lambda_{\min}(p):=\inf\{\lambda(v): v\text{ divisorial valuation reaching }p\},\qquad \lct_p(\I):=\lambda_{\min}(p),
\]
with $\inf\emptyset=+\infty$. By Theorem \ref{thm:divcomp}, for any log-resolution $\pi$,
\[
\lambda_{\min}(p)=\min\{\lambda_E:E\subset D_\pi,\ p\in\pi(E)\},
\]
so that $\lambda_{\min}(p)=+\infty$ off $\pi(D_\pi)$, and also when every divisor over $p$ has $\nu=0$, i.e.\ $p\notin V(\I)$. We write $\lct(\I):=\min_X\lambda_{\min}=\min_{E\subset D_\pi}\lambda_E$ (the minimum is independent of $\pi$ by Theorem \ref{thm:divcomp}).
\end{definition}

The non-exceptional components of $D_\pi$ must be included: a reduced principal ideal has $\lambda=1$ on the strict transform of its zero locus, and this may be the minimum. For $X$ smooth this is the classical resolution formula for the log canonical threshold \cite{Kol97,Mus12}; for singular $X$ it is the Mather log canonical threshold \cite{Ish13}; Theorem \ref{thm:volume} identifies it with a volume exponent defined without a resolution.

\begin{remark}[Real analytic pairs]\label{rem:realpairs}
For a real analytic pair $(X,\I)$ one takes real log-resolutions, the weight is $\gamma_E=\lambda_E/2$, and the volume exponent of $\{K_\I\le\varepsilon\}$ for the $n$-dimensional Hausdorff measure is $\gamma_{\min}$. The persistence modules built with $\gamma$ are those built with $\lambda$ after the global reparametrization $\alpha\mapsto\alpha/2$. Domination of real log-resolutions needs complexification and descent to real points, and weak factorization is not available; this case is treated in \cite{GruR}.
\end{remark}

\section{Volume asymptotics of energy sublevel sets}\label{sec:volume}

Measure theory enters only in this section, stated throughout for $2n$-dimensional volume; it is where the weights $\lambda_E$ are pinned down. When $X$ is singular, every log-resolution in this section is assumed Jacobian-adapted (Definition \ref{def:adapted}); Remark \ref{rem:mather} explains why the local form of the density requires it, and Example \ref{ex:A2fail} what goes wrong without it. A finer study of these asymptotics, beyond the leading term, is carried out in \cite{Gru26}.

\subsection{Why the normalization $\nu_E$}\label{subsec:why}
The definition $\lambda_E=(a_E+1)/\nu_E$ is not one of several equally reasonable normalizations of the ratio $A_E/\nu_E$: it is the unique choice for which $\lambda_E$ is a volume exponent. Take a point of $\Xt$ through which pass the components $E_1,\dots,E_\ell$ of $D_\pi$, and holomorphic coordinates $(u,w)$ with $E_j=\{u_j=0\}$ (Lemma \ref{lem:adapted} below). Two estimates hold there:
\[
K_\I\circ\pi\ \asymp\ \prod_{j=1}^\ell|u_j|^{2\nu_j},\qquad
\pi^*\Hn_X\ \asymp\ \prod_{j=1}^\ell|u_j|^{2a_j}\,d\mathrm{Leb}(u,w),
\]
the first because $\I\cdot\mathcal O_{\Xt}$ is the principal ideal $(\prod u_j^{\nu_j})$ and the energy is a sum of squared moduli, the second because $|u^a|^2$ is the squared modulus of the complex Jacobian, which is the density of the $2n$-dimensional measure. Now substitute, in each variable separately, $t_j=|u_j|^2$ and $x_j=t_j^{\nu_j}=|u_j|^{2\nu_j}$; integrating the angle, $d\mathrm{Leb}(u_j)=\pi\,dt_j$, so that
\[
|u_j|^{2a_j}\,d\mathrm{Leb}(u_j)=\pi\,t_j^{a_j}\,dt_j=C_j\,x_j^{\frac{a_j+1}{\nu_j}-1}\,dx_j=C_j\,x_j^{\lambda_j-1}\,dx_j,\qquad C_j=\tfrac{\pi}{\nu_j}.
\]
The exponent $\lambda_j-1$ is forced: the $\nu_j$ comes from the order of $\I$ (the squares in $K_\I$ and in the Jacobian density cancel), the $a_j+1$ from the Jacobian and the two-dimensional differential. In one variable the sublevel set $\{x_1\le\varepsilon\}$ then has volume $\asymp\int_0^\varepsilon x_1^{\lambda_1-1}dx_1=\varepsilon^{\lambda_1}/\lambda_1$: the divisor $E_1$ contributes to the volume of $U_\varepsilon(\I)$ at the rate $\varepsilon^{\lambda_{E_1}}$, no more and no less. Everything else in this section is the bookkeeping of several divisors through one point.

\begin{example}\label{ex:toy}
Let $X=\C$ and $\I=(x^2)$, so $K_\I=|x|^4$ and $U_\varepsilon(\I)$ is the closed disc of radius $\varepsilon^{1/4}$. The identity is a log-resolution with a single divisor $E=\{0\}$, and
\[
\nu_E=2,\qquad a_E=0,\qquad\lambda_E=\frac{0+1}{2}=\frac12 .
\]
And indeed $\mathcal H^2(U_\varepsilon)=\pi\varepsilon^{1/2}\asymp\varepsilon^{\lambda_E}$: the weight $\tfrac12$ is read off from the area of the disc. Since $V(\I)$ is a point, the local threshold $\lambda_{\min}$ is $\tfrac12$ at the origin and $+\infty$ elsewhere, and the dual complex is a single vertex of weight $\tfrac12$ (Figure~\ref{fig:toy}). For $X=\C^2$ and $\I=(x,y)$ the blow-up of the origin has one divisor with $\nu_E=1$, $a_E=1$, $\lambda_E=2$, and $\mathcal H^4(U_\varepsilon)=\tfrac{\pi^2}{2}\varepsilon^2\asymp\varepsilon^{\lambda_E}$ (the ball of radius $\sqrt\varepsilon$ in $\R^4$). A complex line $\ell$ through the origin has $\mathcal H^2(\ell\cap U_\varepsilon)=\pi\varepsilon$: the exponent $1$ is not a divisorial weight (Remark \ref{rem:kchains}), since the theory concerns $2n$-dimensional volume.
\end{example}

\begin{figure}[ht]
\centering
\begin{tikzpicture}[scale=1]
 \draw[->,black!50] (-1.9,0)--(1.9,0) node[right,lab]{$\Re x$};
 \draw[->,black!50] (0,-1.9)--(0,1.9) node[above,lab]{$\Im x$};
 \foreach \r/\g in {1.5/10,1.0/25,0.5/45}{\draw[fill=black!\g] (0,0) circle (\r);}
 \node[lab] at (1.25,1.25) {$U_\varepsilon$};
 \node[lab] at (0.9,-0.28) {$\varepsilon^{1/4}$};
 \fill (0,0) circle (1.2pt);
 \node[lab,below] at (0,-2.0) {$X=\mathbb C$, $\mathcal I=(x^2)$};
 \draw[->,thick] (2.3,0)--(3.5,0) node[midway,above,lab]{$\pi=\mathrm{id}$};
 \begin{scope}[xshift=6.2cm]
  \node[vt,label={[lab]above:$E=\{0\}$}] at (0,0) {};
  \node[lab] at (0,-0.55) {dual complex: one vertex};
  \node[lab,below] at (0,-2.0) {$\nu_E=2,\ a_E=0,\ \lambda_E=\tfrac12$};
 \end{scope}
\end{tikzpicture}
\caption{Example~\ref{ex:toy}. The sublevel sets $U_\varepsilon$ are discs of radius $\varepsilon^{1/4}$ and $\mathcal H^2(U_\varepsilon)\asymp\varepsilon^{\lambda_E}$ with $\lambda_E=\tfrac12$; the identity resolution has the single divisor $E=\{0\}$.}\label{fig:toy}
\end{figure}

\begin{lemma}\label{lem:adapted}
Let $\pi$ be a log-resolution, Jacobian-adapted if $X$ is singular. Every point of $\Xt$ has a neighbourhood with holomorphic coordinates $(u_1,\dots,u_\ell,w_1,\dots,w_{n-\ell})$ in which the components of $D_\pi$ through the point are $E_j=\{u_j=0\}$, and
\[
K_\I\circ\pi\asymp\prod_{j=1}^\ell|u_j|^{2\nu_{E_j}},\qquad \pi^*\Hn_X\asymp\prod_{j=1}^\ell|u_j|^{2a_{E_j}}\,d\mathrm{Leb}(u,w)
\]
uniformly on compact subsets.
\end{lemma}
\begin{proof}
Principalization gives $f_i\circ\pi=\prod_ju_j^{\nu_{E_j}}h_i$ with $h_i$ analytic and the $h_i$ without common zero near the point, so $\sum|f_i\circ\pi|^2\asymp\prod|u_j|^{2\nu_{E_j}}$. The second estimate is Definition \ref{def:divdata} and Remark \ref{rem:mather}: $\Jac_{\Xt/X}$ is invertible with cosupport in $D_\pi$, hence locally generated by $\prod_ju_j^{a_{E_j}}$ times a unit, so that $\sum|\text{minors}|^2\asymp\prod_j|u_j|^{2a_{E_j}}$.
\end{proof}

\subsection{The monomial model}\label{subsec:monomial}
After the substitution of \S\ref{subsec:why}, the volume of $U_\varepsilon(\I)$ in a chart is, up to bounded factors, the integral of the density $\prod_jx_j^{\lambda_j-1}$ over the monomial sublevel set $\{\prod_jx_j\le\varepsilon\}$ in a cube. This integral is completely explicit and its behaviour is governed by the smallest exponent and by how many times it is attained.

\begin{lemma}\label{lem:monomial}
Let $\lambda_1,\dots,\lambda_r>0$, $\lambda_{\min}=\min_j\lambda_j$ attained $m$ times. Then, as $\varepsilon\to0^+$,
\[
J(\varepsilon)=\int_{[0,1]^r\cap\{\prod x_j\le\varepsilon\}}\prod_jx_j^{\lambda_j-1}\,dx\ \asymp\ \varepsilon^{\lambda_{\min}}\LOG^{m-1}.
\]
\end{lemma}
\begin{proof}
Induction on $r$; order the variables so that $\lambda_r=\max_j\lambda_j$. For $r=1$, $J=\varepsilon^{\lambda_1}/\lambda_1$. For $r\ge2$ integrate first in $x_1,\dots,x_{r-1}$: with $J_{r-1}$ the analogous integral, $J(\varepsilon)=\int_0^1x_r^{\lambda_r-1}J_{r-1}(\varepsilon/x_r)\,dx_r$, where $J_{r-1}(\delta)=A$ is constant for $\delta\ge1$ and $J_{r-1}(\delta)\asymp\delta^{\lambda'_{\min}}(1+|\log\delta|)^{m'-1}$ uniformly for $0<\delta\le1$ by induction (the asymptotic statement for $\delta\to0$ together with continuity and positivity on $[\delta_0,1]$) ($\lambda'_{\min}=\lambda_{\min}$ since $\lambda_r$ is a maximum). Splitting at $x_r=\varepsilon$ and substituting $x_r=e^{-w}$, $L=\LOG$,
\[
J(\varepsilon)\asymp\varepsilon^{\lambda_r}+\varepsilon^{\lambda_{\min}}\int_0^Le^{-(\lambda_r-\lambda_{\min})w}(1+L-w)^{m'-1}\,dw.
\]
If $\lambda_r>\lambda_{\min}$ the integral is $\asymp L^{m'-1}$ for large $L$, $m=m'$, and $\varepsilon^{\lambda_r}=o(\varepsilon^{\lambda_{\min}})$. If $\lambda_r=\lambda_{\min}$ the integral is $\asymp L^{m'}$, $m=m'+1\ge2$, and the term $\varepsilon^{\lambda_r}=\varepsilon^{\lambda_{\min}}$ is absorbed because $L^{m-1}\to\infty$. In both cases $J\asymp\varepsilon^{\lambda_{\min}}L^{m-1}$.
\end{proof}

\begin{remark}\label{rem:monogeom}
Pass to logarithmic coordinates $t_j=-\log x_j\ge0$; the density $\prod x_j^{\lambda_j-1}dx$ becomes $e^{-\sum_j\lambda_jt_j}dt$ and the domain becomes the half-space $\{\sum_jt_j\ge L\}$, $L=\LOG$, in the positive orthant. The linear functional $\sum_j\lambda_jt_j$ restricted to the simplex $\{\sum t_j=L\}$ attains its minimum $\lambda_{\min}L$ on the face spanned by the coordinate directions with $\lambda_j=\lambda_{\min}$; that face has dimension $m-1$, and the mass of $e^{-\sum\lambda_jt_j}$ beyond the hyperplane concentrates along it, producing $e^{-\lambda_{\min}L}$ times a polynomial factor $L^{m-1}$ from the volume of the face: a single smallest exponent gives $\varepsilon^{\lambda_{\min}}$, two equal smallest exponents give an extra factor $\LOG$, three give $\LOG^2$, and so on, so $m-1$ is the dimension of the degeneracy of the minimum in the monomial model. This is the geometric content of the resonance multiplicity $m(p)$ defined below.
\end{remark}

\subsection{Volume asymptotics}\label{subsec:volthm}

\begin{theorem}\label{thm:volume}
Let $(X,\I)$ be an analytic pair, $\pi$ a log-resolution (Jacobian-adapted if $X$ is singular), and $p\in V(\I)$. Let $m_\pi(q)\ge1$ denote, for $q\in V(\I)$, the largest number of components of $D_\pi$ of weight $\lambda_{\min}(q)$ passing through a common point of $\pi^{-1}(q)$. Let $B$ be a compact neighbourhood of $p$ in $X$ small enough that
\[
\lambda_{\min}(q)\ge\lambda_{\min}(p)\ \text{ for } q\in B,\qquad m_\pi(q)\le m_\pi(p)\ \text{ for } q\in B\text{ with }\lambda_{\min}(q)=\lambda_{\min}(p).
\]
Then, as $\varepsilon\to0^+$,
\[
\Hn\bigl(B\cap U_\varepsilon(\I)\bigr)\asymp\varepsilon^{\lambda_{\min}(p)}\LOG^{m(p)-1},\qquad m(p):=m_\pi(p).
\]
In particular
\[
\lambda_{\min}(p)=\lct_p(\I)=\sup\{\alpha:\Hn(B\cap U_\varepsilon)=O(\varepsilon^\alpha)\},
\]
and $\lambda_{\min}(p)=\lct_p(\I)$, which for smooth $X$ is the log canonical threshold of $\I$ at $p$ and for singular $X$ the Mather log canonical threshold \cite{Ish13}, is the volume-growth exponent of $U_\varepsilon(\I)$.
\end{theorem}
\begin{proof}
Such $B$ exist: $\lambda_{\min}$ takes finitely many values, all in $\Gamma_\pi(X,\I)\cup\{+\infty\}$, and $\{\lambda_{\min}<\alpha\}=\bigcup_{\lambda_E<\alpha}\pi(E)$ is closed since $\pi$ is proper, see Corollary \ref{cor:semicont}, whose proof does not use this theorem; and, once $\lambda_{\min}\ge\lambda_{\min}(p)$ on $B$, the set of $q\in B$ with $\lambda_{\min}(q)=\lambda_{\min}(p)$ and $m_\pi(q)\ge k$ is $B\cap\bigcup_\sigma\pi(E_\sigma)$, the union over the sets $\sigma$ of $k$ components of weight $\lambda_{\min}(p)$ with $E_\sigma=\bigcap_{E\in\sigma}E\neq\emptyset$, which is closed and contains $p$ iff $k\le m_\pi(p)$. Fix such a $B$.
$\pi^{-1}(B)$ is compact and covered by finitely many charts of Lemma \ref{lem:adapted}. In a chart meeting $E_1,\dots,E_\ell$ the volume of $\{K_\I\circ\pi\le\varepsilon\}$ is $\asymp\int_{\prod|u_j|^{2\nu_j}\le\varepsilon}\prod|u_j|^{2a_j}\,d\mathrm{Leb}(u)$, which after $x_j=|u_j|^{2\nu_j}$, $|u_j|^{2a_j}\,d\mathrm{Leb}(u_j)=C_j\,x_j^{(a_j+1)/\nu_j-1}dx_j=C_j\,x_j^{\lambda_j-1}dx_j$ (with $C_j=\pi/\nu_j$ of \S\ref{subsec:why}), becomes the integral of Lemma \ref{lem:monomial} with exponents $\lambda_j$ (divisors with $\nu_j=0$ impose no constraint and only contribute bounded factors). The integral over a whole chart box is $\asymp\varepsilon^{\min_j\lambda_j}\LOG^{m_{\mathrm{chart}}-1}$; since $\pi^{-1}(B)$ meets the chart in a subset of the box, this gives an \emph{upper} bound for the contribution of the chart. Now, if a divisor $E$ meets $\pi^{-1}(B)$ then $\pi(E)\cap B\neq\emptyset$, so some $q\in B$ has $\lambda_{\min}(q)\le\lambda_E$, whence $\lambda_E\ge\lambda_{\min}(p)$ by the hypothesis on $B$: every chart has $\min_j\lambda_j\ge\lambda_{\min}(p)$. Moreover, in a chart centred at $y$ with $\min_j\lambda_j=\lambda_{\min}(p)$, the components of weight $\lambda_{\min}(p)$ through $y$ all pass through a common point of $\pi^{-1}(q)$, $q=\pi(y)\in B$, so $\lambda_{\min}(q)=\lambda_{\min}(p)$ and $m_{\mathrm{chart}}\le m_\pi(q)\le m_\pi(p)$ by the second hypothesis on $B$; a chart centred at a point of $\pi^{-1}(p)$ where $m_\pi(p)$ components of weight $\lambda_{\min}(p)$ meet realizes $m_{\mathrm{chart}}=m_\pi(p)$. Summing the finitely many charts gives the upper bound $\lesssim\varepsilon^{\lambda_{\min}(p)}\LOG^{m_\pi(p)-1}$. For the lower bound, $B$ is a neighbourhood of $p$ and $\pi$ is continuous, so $\pi^{-1}(B)$ contains a small box around the centre $y\in\pi^{-1}(p)$ of the chart realizing $m_\pi(p)$; Lemma \ref{lem:monomial} applies to the integral over this smaller box with the same exponent and the same power of the logarithm, since both are determined by the behaviour of the integrand near $\{x=0\}$. This gives the matching lower bound. The last assertion is immediate from it.
\end{proof}

\begin{corollary}\label{cor:semicont}
For every $\alpha>0$ the set $Z_\alpha:=\{p\in X:\lambda_{\min}(p)<\alpha\}$ equals $\bigcup_{\lambda_E<\alpha}\pi(E)$ and is a closed subanalytic subset of $X$ contained in $V(\I)$; $\lambda_{\min}$ is lower semicontinuous, and takes finitely many values, all in $\Gamma_\pi(X,\I)\cup\{+\infty\}$.
\end{corollary}
\begin{proof}
$\lambda_{\min}(p)<\alpha$ iff some $E$ with $\lambda_E<\alpha$ has $p\in\pi(E)$; $\pi(E)$ is closed since $\pi$ is proper, and $\pi(E)\subset V(\I)$ when $\nu_E>0$.
\end{proof}

\begin{corollary}\label{cor:mp}
The integer $m(p)=m_\pi(p)$ of Theorem \ref{thm:volume} does not depend on the (Jacobian-adapted) log-resolution: it is one plus the exponent of the logarithmic correction to the volume of $B\cap U_\varepsilon(\I)$ for every sufficiently small $B$, and $\varepsilon\mapsto\Hn(B\cap U_\varepsilon(\I))$ is defined without any resolution. If $\Phi\colon X\to Y$ is a bi-Lipschitz subanalytic homeomorphism with $K_\J\circ\Phi\asymp K_\I$, then $\lambda_{\min}(\Phi(p))=\lambda_{\min}(p)$ and $m(\Phi(p))=m(p)$ for all $p$.
\end{corollary}
\begin{proof}
For the first, given two log-resolutions choose $B$ small enough for both; the two-sided estimates then agree, and both the exponent and the logarithmic order are read off from $\Hn(B\cap U_\varepsilon)$. For the second, $\Phi$ distorts $\Hn$ by bounded factors and maps $U_\varepsilon(\I)$ into $U_{C\varepsilon}(\J)$ and $U_{c\varepsilon}(\J)$ into it, so the two-sided estimates for the two pairs agree, and both the exponent and the logarithmic order are read off from them.
\end{proof}

The pair $(\lambda_{\min}(p),m(p))$ is thus a two-parameter invariant of the germ of $(X,\I)$ at $p$ obtained directly from the volume asymptotics; for hyperplane arrangements it is the pair (threshold, multiplicity) of \cite{KW26}.

\subsection{The resonance multiplicity}\label{subsec:mp}
The combinatorial description of $m(p)$ on a resolution is the first place where the incidence structure of the divisors, not only the list of their weights, enters an invariant of the pair. Recall from Section \ref{sec:dual} the notation $E_\sigma=\bigcap_{i\in\sigma}E_i$ for $\sigma$ a set of components of $D_\pi$.

The next lemma computes how the weight changes under a single blow-up; we call the resulting two-sided estimate the \emph{mediant inequality}.

\begin{lemma}\label{lem:mediant}
Let $\pi\colon Y\to X$ be a log-resolution, $Z\subset Y$ a smooth irreducible closed subset having simple normal crossings with $D_\pi$, $\sigma=\{j:Z\subseteq E_j\}$ (so $s=|\sigma|\le r=\codim Z$), and $\rho\colon Y'\to Y$ the blow-up along $Z$ with exceptional divisor $F$. Let $\ord_Z(\Jac_{Y/X})$ be the order of the Jacobian ideal along $Z$ (the largest $k$ with $\Jac_{Y/X}\subseteq I_Z^k$ near a generic point of $Z$) and put
\[
e_Z:=\ord_Z(\Jac_{Y/X})-\sum_{j\in\sigma}a_j\ \ge0 .
\]
Then $\nu_F=\sum_{j\in\sigma}\nu_j$, $a_F=\sum_{j\in\sigma}a_j+(r-1)+e_Z$, and, when $\nu_F>0$,
\[
\min_{j\in\sigma,\ \nu_j>0}\lambda_j\ \le\ \lambda_F=\sum_{j\in\sigma,\ \nu_j>0}\frac{\nu_j}{\nu_F}\lambda_j+\frac{(r-s)+\sum_{j\in\sigma,\ \nu_j=0}(a_j+1)}{\nu_F}+\frac{e_Z}{\nu_F}.
\]
The two remainders are non-negative. The first vanishes iff $r=s$ and $\nu_j>0$ for all $j\in\sigma$. The excess $e_Z$ vanishes if $\Jac_{Y/X}$ is invertible in a neighbourhood of a generic point of $Z$; in particular $e_Z=0$ for every snc centre when $\pi$ is Jacobian-adapted or $X$ is smooth. Consequently:
\begin{enumerate}[label=\textup{(\roman*)},leftmargin=2em]
\item (\emph{lower bound, always}) $\lambda_F\ge\min_{j\in\sigma,\nu_j>0}\lambda_j$ on every log-resolution;
\item (\emph{two-sided bound, adapted log-resolutions}) if $\pi$ is Jacobian-adapted, $r=s$ (i.e.\ $Z$ is a component of $\bigcap_{j\in\sigma}E_j$) and all $\nu_j>0$, then $\lambda_F$ is the $\nu$-weighted average of the $\lambda_j$, $j\in\sigma$, and $\min_{j\in\sigma}\lambda_j\le\lambda_F\le\max_{j\in\sigma}\lambda_j$.
\end{enumerate}
Divisors $E_i$ not containing $Z$ keep their data on the strict transform. If $\sigma=\emptyset$ or all $\nu_j=0$ then $\nu_F=0$ and $\lambda_F=+\infty$. On a non-adapted resolution the upper bound in \textup{(ii)} can fail (Example \ref{ex:A2fail}).
\end{lemma}
\begin{proof}
$\rho^*E_j=E'_j+F$ for $j\in\sigma$ and $\rho^*E_i=E'_i$ otherwise, and $\I\cdot\mathcal O_Y=\mathcal O_Y(-\sum\nu_iE_i)$ is invertible, so $\nu_F=\ord_F(\rho^*\I\mathcal O_Y)=\sum_{j\in\sigma}\nu_j$. For $a_F$ use the chain rule $\Jac_{Y'/X}=\Jac_{Y'/Y}\cdot\rho^*\Jac_{Y/X}$ (Lemma \ref{lem:adaptdom}(i)) with $\Jac_{Y'/Y}=(\det D\rho)=\mathcal O_{Y'}(-(r-1)F)$: $a_F=(r-1)+\ord_F(\rho^*\Jac_{Y/X})$, and $\ord_F(\rho^*\mathfrak a)=\ord_Z(\mathfrak a)$ for every ideal $\mathfrak a$, since $\rho^*I_Z=\mathcal O_{Y'}(-F)$ and $\ord_F$ is the $I_Z$-adic order at the generic point of $Z$ (in blow-up coordinates $\rho^*g$ has order exactly $k$ along $F$ for $g\in I_Z^k\setminus I_Z^{k+1}$, and the order of an ideal is the minimum over its elements). Now $\Jac_{Y/X}\subseteq\bigcap_{j\in\sigma}\mathcal O_Y(-a_jE_j)=\mathcal O_Y(-\sum_{j\in\sigma}a_jE_j)$ near $Z$ (the $E_j$ are distinct components of a simple normal crossings divisor and $\mathcal O_{Y,\eta_Z}$ is a regular local ring), and $\ord_Z(\prod_{j\in\sigma}u_j^{a_j})=\sum_{j\in\sigma}a_j$ because the $u_j$, $j\in\sigma$, are part of a regular system of parameters for $I_Z$; hence $\ord_Z(\Jac_{Y/X})\ge\sum_{j\in\sigma}a_j$, i.e.\ $e_Z\ge0$. If $\Jac_{Y/X}$ is invertible near a generic point $\eta$ of $Z$, it is generated there by one function $g$ whose zero locus is a divisor contained in $\Exc(\pi)\subseteq D_\pi$ (off $\Exc(\pi)$ the map $\pi$ is a local isomorphism, so $X$ is smooth there and $\Jac_{Y/X}=\mathcal O_Y$); the components of $D_\pi$ through $\eta$ are exactly the $E_j$, $j\in\sigma$ (the others meet $Z$ in a nowhere dense subset), so $g=\prod_{j\in\sigma}u_j^{a_j}\cdot(\text{unit})$ near $\eta$ and $e_Z=0$. The formula for $\lambda_F$ follows by substituting $a_j+1=\nu_j\lambda_j$ for the $j$ with $\nu_j>0$ and keeping the terms $a_j+1$ for the $j$ with $\nu_j=0$. (i) holds because the first sum is a convex combination and both remainders are $\ge0$; (ii) because both remainders vanish.
\end{proof}

\begin{remark}\label{rem:mediantinterp}
The formula for $\lambda_F$ has three terms with different meanings.

\emph{The convex part.} When $r=s$, i.e.\ when the centre $Z$ is a stratum $E_\sigma$ of the boundary, all $\nu_j>0$ for $j\in\sigma$, and the model is adapted along $Z$, both remainders vanish and
\[
\lambda_F=\sum_{j\in\sigma}\frac{\nu_j}{\sum_{i\in\sigma}\nu_i}\,\lambda_j ,
\]
a weighted average of the weights of the divisors containing $Z$, with weights proportional to the vanishing orders. This is why $\lambda_F$ is a \emph{mediant}: for two divisors, $\lambda_F=(a_1+a_2+2)/(\nu_1+\nu_2)$ is the mediant of the fractions $(a_1+1)/\nu_1$ and $(a_2+1)/\nu_2$, and the mediant of two fractions lies between them. In that case $\min_j\lambda_j\le\lambda_F\le\max_j\lambda_j$, and this two-sided bound is the entire numerical input of the invariance results of Section \ref{sec:dual}: it says that a divisor created by blowing up a stratum has a weight that is neither smaller than the smallest nor larger than the largest weight of the divisors through the stratum, so that it enters every sublevel complex no earlier than the earliest and no later than the latest of them.

\emph{The transversal part.} When $r>s$ the centre $Z$ has codimension $r-s$ inside the stratum $E_\sigma$, and the term $(r-s)/\nu_F$ is the contribution of these $r-s$ transversal directions to the discrepancy ($K_{Y'/Y}=(r-1)F$ instead of $(s-1)F$) with no corresponding contribution to $\nu_F$; the divisors of $\sigma$ with $\nu_j=0$ (weight $+\infty$) contribute in the same way, through their Jacobian orders $a_j+1$ only. The remainder is non-negative, so the lower bound survives; the upper bound may fail, and indeed a leaf attached to a curve of a surface resolution by blowing up a point on it (Theorem \ref{thm:surface}) has weight $\lambda_i+1/\nu_i>\lambda_i$. That the upper bound is only needed when $r=s$, and holds exactly then on adapted log-resolutions, is what makes the two cases of Lemma \ref{lem:dualblowup} match the two conclusions of the present lemma.

\emph{The excess.} The third term $e_Z/\nu_F$ measures how far $\Jac_{Y/X}$ is from being monomial along $Z$; it is the contribution of the singularities of $X$ that the model $Y$ has not yet ``seen'' along $Z$. It vanishes on adapted log-resolutions and can be positive otherwise, even when $r=s$: for $A_2$ the two curves through the point have $\lambda=2$ and the excess is $1$, so the new curve has $\lambda=5/2$ (Example \ref{ex:A2fail}). Unlike the transversal part it is not attached to a leaf: it raises the weight of a divisor that subdivides an edge, and this is what breaks invariance. On an adapted log-resolution the excess is absent along every snc centre, and remains absent after the blow-up (Lemma \ref{lem:adaptdom}); this is the reason for working within the adapted class in Sections \ref{sec:dual} and \ref{sec:surfaces}, and it is a restriction on the model, not on the pair.
\end{remark}

\begin{proposition}\label{prop:mp}
For every Jacobian-adapted log-resolution $\pi$ and every $p\in V(\I)$,
\[
m(p)=\max\bigl\{|\sigma|:\ E_\sigma\cap\pi^{-1}(p)\neq\emptyset\ \text{ and }\ \lambda_E=\lambda_{\min}(p)\text{ for all }E\in\sigma\bigr\}.
\]
Equivalently, $m(p)-1$ is the largest dimension of a simplex of the dual complex $\Delta(\pi)$ (Definition \ref{def:dual}) all of whose vertices have weight $\lambda_{\min}(p)$ and whose stratum meets $\pi^{-1}(p)$. In particular $1\le m(p)\le n$. The right-hand side is unchanged by snc blow-ups of Jacobian-adapted log-resolutions. On a non-adapted resolution the right-hand side is still defined, but it can be strictly larger than $m(p)$.
\end{proposition}
\begin{proof}
The first assertion is the definition of $m_\pi(p)$ in Theorem \ref{thm:volume} read in the language of strata: a set of components of weight $\lambda_{\min}(p)$ passes through a common point of $\pi^{-1}(p)$ iff the corresponding stratum $E_\sigma$ meets $\pi^{-1}(p)$, and by simple normal crossings $|\sigma|\le n$. For the invariance, let $\rho$ be an snc blow-up of an adapted log-resolution with centre $Z$ and new divisor $F$. By Lemma \ref{lem:mediant}, on an adapted log-resolution $\lambda_F$ is the mediant expression with zero excess, and $\lambda_F=\lambda_{\min}(p)$ forces $r=s$, all $\nu_j>0$ and $\lambda_j=\lambda_{\min}(p)$ for all $j\in\sigma$ (the remainder is positive otherwise, and a convex combination with positive coefficients of numbers $\ge\lambda_{\min}(p)$ equals $\lambda_{\min}(p)$ only if all of them do); in that case $Z$ is a stratum $E_\sigma$ and by Lemma \ref{lem:dualblowup}(a) the dual complex is stellarly subdivided at $\tau_Z$, which does not change the maximal dimension of a simplex with all vertices of weight $\lambda_{\min}(p)$ meeting $\pi^{-1}(p)$ (a simplex of the subdivision containing $v_F$ has the same dimension as the simplex it came from, and its stratum lies over the same points). If $\lambda_F>\lambda_{\min}(p)$, then either $Z$ is not a stratum all of whose divisors have weight $\lambda_{\min}(p)$, in which case the strata of the required kind are strict transforms of the old ones and the maximum is unchanged; or $Z$ is such a stratum, which on an adapted log-resolution is impossible by the equality case just described. It is exactly this last alternative that occurs on non-adapted log-resolutions: there $Z=E_\sigma$ with all $\lambda_j=\lambda_{\min}(p)$ but positive excess gives $\lambda_F>\lambda_{\min}(p)$, and the stratum $E_\sigma$ disappears from the count without a replacement. Finally, adapted resolutions compute the intrinsic $m(p)$ by Corollary \ref{cor:mp}.
\end{proof}
On $A_2$, the minimal resolution's count is $2$ while the intrinsic $m(0)$ is $1$ (Example \ref{ex:A2fail}).

\begin{remark}[Detection versus classification]\label{rem:detection}
What Example~\ref{ex:Anexplicit} establishes is a statement of \emph{detection}: on the family $\{A_n\}_{n\ge1}$ the map $n\mapsto\DPH_0(A_n,\mathfrak m)$ is injective, whereas $\lct$ and $m(0)$ are constant for $n\ge3$. It does \emph{not} establish that $\DPH$ is a complete invariant, for three reasons. (i) Only the family $A_n$ was computed; the other simple singularities and the non-rational ones are not treated, so nothing is claimed about whether $\DPH$ separates them from each other or from $A_n$. (ii) By Remark~\ref{rem:disconnected}, on a normal surface germ the module is a function of the weighted dual graph alone, so two germs with the same weighted dual graph have the same module, whatever their analytic moduli; $\DPH$ cannot see information that is not encoded in the weights and incidences. (iii) Even for the pair, the module is the image of the filtered complex under a forgetful functor (Remark~\ref{rem:forgets}), so inequivalent filtered complexes with the same homology give the same module. A classification statement would require an inverse, from the module to the germ, which the paper does not provide.
\end{remark}

\begin{remark}\label{rem:forgets}
The resonance multiplicity is a datum of the filtered complex, not of its homology. On $X=\C^2$ take $\I=(x)$ and $\I'=(xy)$, with the identity as log-resolution. For $\I$ the dual complex is the single vertex $v_x$ of weight $1$; for $\I'$ it is the edge $v_xv_y$, both vertices of weight $1$ (here $\nu=1$, $a=0$). Both sublevel complexes are contractible for $\alpha\ge1$ and empty before, so the two persistence modules are equal ($\DPH_0=\Z$ from $1$ and nothing else), whereas $m(0)=1$ for $\I$ and $m(0)=2$ for $\I'$: $\Hn\{|x|^2\le\varepsilon\}\asymp\varepsilon$ near $0$, but $\Hn\{|xy|^2\le\varepsilon\}\asymp\varepsilon|\log\varepsilon|$.
\begin{center}
\begin{tikzpicture}[x=1cm,y=1cm]
 \node[vt,label={[lab]below:$v_x\ (1)$}] at (0,0) {};
 \node[lab] at (0,-1.0) {$\I=(x)$,\ $m(0)=1$};
 \begin{scope}[xshift=4.5cm]
  \draw[eon] (0,0)--(1.4,0);
  \node[vt,label={[lab]below:$v_x\ (1)$}] at (0,0) {};
  \node[vt,label={[lab]below:$v_y\ (1)$}] at (1.4,0) {};
  \node[lab] at (0.7,-1.0) {$\I'=(xy)$,\ $m(0)=2$};
 \end{scope}
\end{tikzpicture}
\end{center}
Thus $\DPH$ is a projection of the filtered complex and forgets part of it; the pair $(\lambda_{\min},m)$ is read from the filtered complex itself (Proposition \ref{prop:mp}), not from $\DPH$. Likewise the values of $\Gamma_\pi(X,\I)$ are only candidate critical values: the \emph{effective} critical values, those $\alpha$ for which $\DPH^{\alpha-\delta}\to\DPH^{\alpha}$ is not an isomorphism for all sufficiently small $\delta>0$, form a subset of the spectrum (the smallest $\alpha$ with $\DPH^\alpha\ne0$ is always effective), and in the example above only the first weight is effective.
\end{remark}

\begin{remark}\label{rem:zeta}
The indicator function of a compact neighbourhood $B$ has no regularity at $\partial B$, so we use a cutoff. Let $\varphi\ge0$ be the restriction to $X$ of a function in $C_c^\infty(\C^N)$, supported in a neighbourhood $B$ of $p$ as in Theorem \ref{thm:volume} and positive near $p$, and consider $\zeta_\varphi(s)=\int_X\varphi\,K_\I^{\,s}\,d\Hn$, $\Re s>0$. On a Jacobian-adapted resolution, $\varphi\circ\pi$ is a smooth function on $\Xt$, and in a chart of Lemma \ref{lem:adapted} the integrand $K_\I^s\,d\Hn$ pulls back to $\prod_j|u_j|^{2\nu_js+2a_j}\,A(u,w;s)\,d\mathrm{Leb}$, where $A=(\varphi\circ\pi)\,h^{s}g$ with $K_\I\circ\pi=\prod_j|u_j|^{2\nu_j}h$ and $\pi^*\Hn=\prod_j|u_j|^{2a_j}g\,d\mathrm{Leb}$ for positive real analytic $h,g$ (adaptedness), so $A$ is smooth with compact support in $(u,w)$, holomorphic in $s$, and positive near the points over $p$ for real $s$; the integral of $|u_j|^{2\nu_js+2a_j}$ over the unit disc is $\pi/(\nu_js+a_j+1)$, with a simple pole at $s=-\lambda_j$, and the standard partition-of-unity argument (see \cite[Ch.~4]{Wat09} for the real case, the complex case being identical) (Taylor expansion of the smooth amplitude in $u_j,\bar u_j$, of which only the rotation-invariant terms $|u_j|^{2k}$ survive the angular integration) gives a meromorphic continuation of $\zeta_\varphi$ to $\C$ whose poles are among the numbers $-\lambda_E-k/\nu_E$, $k\in\Z_{\ge0}$, $E\subset D_\pi$ with $\nu_E>0$ (on an adapted resolution both $K_\I\circ\pi$ and the pulled-back density are monomials times positive real analytic functions), whose largest pole is $s=-\lambda_{\min}(p)$, and whose order at that pole is the largest number of factors with $\lambda_j=\lambda_{\min}(p)$ occurring in one chart over a point of $\pi^{-1}(p)$, i.e.\ $m(p)$ (Proposition \ref{prop:mp}); there is no cancellation between charts, since $\varphi\ge0$ and the leading coefficients at the real pole are positive. Without a cutoff, i.e.\ for the indicator of $B$, the boundary of $B$ would have to be controlled (for instance $B$ subanalytic and put in normal crossing with $D_\pi$) and we make no claim. Adaptedness is needed for the local form of the integrand exactly as in Theorem \ref{thm:volume}: on a non-adapted resolution the density carries an extra non-monomial factor and the pole order read from the divisors through a point can overestimate the true one (Example \ref{ex:A2fail}). The pair $(\lambda_{\min}(p),m(p))$ is thus the largest pole of $\zeta_\varphi$ and its order.
\end{remark}

\begin{remark}\label{rem:tauberian}
Theorem \ref{thm:volume} is proved directly, through Lemma \ref{lem:monomial}, and does not use the zeta function; conversely, Remark \ref{rem:zeta} is a statement about a meromorphic continuation and does not by itself give the two-sided estimate of the volume. The two are related by the Mellin transform $\zeta_\varphi(s)=\int_0^\infty\varepsilon^s\,d\mu_\varphi(\varepsilon)$, $\mu_\varphi(\varepsilon)=\int_{U_\varepsilon(\I)}\varphi\,d\Hn$: the estimate $\mu_\varphi\asymp\varepsilon^{\lambda}\LOG^{m-1}$ (Theorem \ref{thm:volume}, for $\varphi$ bounded and bounded below near $p$)  determines the abscissa of the first pole and bounds its order (an Abelian argument), whereas recovering the volume estimate from the pole requires a Tauberian theorem with remainder. The persistence module of Section \ref{sec:dual} is built on the weights, not on the poles.
\end{remark}

\begin{remark}\label{rem:kchains}
Theorem \ref{thm:volume} does not extend to the $k$-dimensional Hausdorff measure of $k$-dimensional subsets for $k<2n$. In adapted coordinates the density $\prod|u_j|^{2a_j}$ is the $2n$-dimensional Jacobian of $\pi$; the $k$-dimensional Jacobian along a $k$-plane is the product of $k$ singular values of $D\pi$ and can vanish to much lower order. For $X=\C^2$, $\I=(x,y)$: the blow-up has one divisor with $\nu=1$, $a=1$, $\lambda=2$, and $\mathcal H^4(U_\varepsilon)=\tfrac{\pi^2}2\varepsilon^2$ as predicted; but a complex line through the origin has $\mathcal H^2(\ell\cap U_\varepsilon)=\pi\varepsilon$, since $\pi$ is an isometry along the strict transform of the line. In general a real $k$-plane through the origin decays like $\varepsilon^{k/2}$. Any statement relating $\lambda_E$ to decay rates of lower-dimensional chains would need the $k$-Jacobian orders $a^{(k)}_E\le a_E$.
\end{remark}

\begin{remark}\label{rem:arrangements}
For $\I=(f)$ with $f=\prod_iL_i^{s_i}$ a product of linear forms on $\C^n$, a wonderful log-resolution has one divisor $D_W$ for each flat $W$ of the intersection lattice, with $\nu_{D_W}=s(W)=\sum_{W\subseteq H_i}s_i$, $a_{D_W}=\codim W-1$, hence $\lambda_{D_W}=\codim W/s(W)$, and $\pi(D_W)=W$. So $\lambda_{\min}(p)=\min\{\codim W/s(W):p\in W\}$, $Z_\alpha$ is the union of the flats with $\codim W/s(W)<\alpha$, and, since on the wonderful model $D_W\cap D_{W'}\neq\emptyset$ iff $W,W'$ are comparable, $m(p)$ is the length of the longest chain of flats through $p$ attaining the minimum, recovering the pair (threshold, multiplicity) of \cite{KW26}.
\end{remark}

\section{The local threshold}\label{sec:threshold}

By Definition \ref{def:localthr}, $\lambda_{\min}\colon X\to(0,+\infty]$ is defined through divisorial valuations alone. This section proves that it is computed on any log-resolution. For singular $X$ only the lower mediant bound is used (Remark \ref{rem:singX}); no Jacobian-adaptedness is needed.

\begin{theorem}\label{thm:divcomp}
For every log-resolution $\pi$ and every $p\in X$,
\[
\lambda_{\min}(p)=\delta_\pi(p):=\min\{\lambda_E:E\subset D_\pi,\ p\in\pi(E)\}.
\]
\end{theorem}
In particular $\lambda_{\min}$ does not depend on $\pi$.

\begin{lemma}\label{lem:oneblowup}
With the notation of Lemma \ref{lem:mediant}, $\delta_{\pi\circ\rho}(p)=\delta_\pi(p)$ for every $p\in X$.
\end{lemma}
\begin{proof}
A strict transform $E'_i$ satisfies $p\in(\pi\rho)(E'_i)$ iff $p\in\pi(E_i)$: the image of $E'_i$ is contained in that of $E_i$; conversely if $y\in E_i$ with $\pi(y)=p$, then either $y\notin Z$ and $y$ lifts to $E'_i$, or $y\in Z$ and $\rho^{-1}(y)\cap E'_i\neq\emptyset$ by normal crossings (the projectivized normal space is a complex projective space and $E'_i$ meets it in a linear subspace). Weights of strict transforms are unchanged. (If $Z$ is itself a component $E_j$, then $\rho$ is an isomorphism, $F=E_j$, and nothing changes.) The divisor $F$ has $p\in(\pi\rho)(F)$ iff $p\in\pi(Z)$, in which case $p\in\pi(E_j)$ for every $j\in\sigma$, and $\lambda_F\ge\min_{j\in\sigma,\nu_j>0}\lambda_j\ge\delta_\pi(p)$ by Lemma \ref{lem:mediant}(i) (if $\nu_F=0$, $\lambda_F=+\infty$). So adjoining $F$ does not lower the minimum. Only the lower bound was used, so the statement holds on every log-resolution, adapted or not.
\end{proof}
\begin{proof}[Proof of Theorem \ref{thm:divcomp}]
\emph{Claim:} $\delta_Y(p)\le\delta_W(p)$ for any two log-resolutions $\pi_Y,\pi_W$. By Proposition \ref{prop:domination} there are a log-resolution $W_1$ obtained from $Y$ by snc blow-ups and a proper birational $\psi\colon W_1\to W$; by Lemma \ref{lem:oneblowup} iterated, $\delta_{W_1}(p)=\delta_Y(p)$. Let $F\subset D_W$ with $p\in\pi_W(F)$ and let $G\subset W_1$ be its strict transform; $\psi$ is an isomorphism at the generic point of $G$, so $G\subset D_{W_1}$ (it lies in the preimage of $\Supp\pi_W^{-1}\I$ or in the exceptional locus, as $F$ does), $\ord_G=\ord_F$, $\lambda_G=\lambda_F$ by Proposition \ref{prop:valinv}(ii), and $\pi_{W_1}(G)=\pi_W(F)\ni p$. Hence $\delta_{W_1}(p)\le\delta_W(p)$, which is the claim. Exchanging the roles gives $\delta_Y=\delta_W$.

Now every $E\subset D_Y$ with $p\in\pi(E)$ gives a valuation $\ord_E$ reaching $p$ with $\lambda=\lambda_E$, so $\lambda_{\min}(p)\le\delta_Y(p)$. Conversely if $v=q\ord_F$, $F\subset D_W$, reaches $p$ then $p\in\pi_W(F)$ (Proposition \ref{prop:valinv}(ii)) and $\lambda(v)=\lambda_F\ge\delta_W(p)=\delta_Y(p)$.
\end{proof}

\begin{remark}\label{rem:singX}
When $X$ is singular, the only numerical input into Theorem \ref{thm:divcomp} is Lemma \ref{lem:mediant}(i), the lower bound $\lambda_F\ge\min_{j\in\sigma,\nu_j>0}\lambda_j$, which holds on every log-resolution because the excess $e_Z$ is non-negative. Consequently $\lambda_{\min}$ and the sets $Z_\alpha$ are computed on arbitrary log-resolutions, adapted or not. Sections \ref{sec:volume}, \ref{sec:dual} and \ref{sec:surfaces} use the exact formula of Lemma \ref{lem:mediant}, hence adaptedness. As a check, in Example \ref{ex:A2fail} the threshold $\lambda_{\min}(0)=2$ is the same on the minimal resolution and on its adapted blow-up, although the filtered dual complexes differ.
\end{remark}

\section{The filtered dual complex}\label{sec:dual}

\subsection{Weighted dual complexes}
Let $\pi$ be a log-resolution, Jacobian-adapted if $X$ is singular (Definition \ref{def:adapted}; automatic if $X$ is smooth), and $D_\pi=\bigcup_{i\in I}E_i$. For $\sigma\subseteq I$ write $E_\sigma=\bigcap_{i\in\sigma}E_i$; by simple normal crossings each connected component of $E_\sigma$ is a smooth closed complex submanifold of codimension $|\sigma|$ (a closed stratum). The local computations of Lemma \ref{lem:dualblowup} take place in $\C^n$, and the exceptional divisor of an snc blow-up is a non-empty divisor. From here to the end of the paper, ``log-resolution'' and ``model'' mean Jacobian-adapted log-resolution when $X$ is singular (Example \ref{ex:A2fail}).

\begin{definition}\label{def:dual}
The \emph{dual complex} $\Delta(\pi)$ is the finite $\Delta$-complex with one vertex $v_i$ for each irreducible component $E_i$ and one open $(|\sigma|-1)$-cell for each connected component of $E_\sigma=\bigcap_{i\in\sigma}E_i$, with face maps induced by inclusions of strata. We use the $\Delta$-complex convention in which distinct cells may have the same set of vertices; in particular two components of $E_i\cap E_j$ give two distinct edges between $v_i$ and $v_j$. The weight $\lambda\colon\{v_i\}\to(0,+\infty]$ is $\lambda(v_i)=\lambda_{E_i}$. For $\alpha>0$ the sublevel complex $\Delta(\pi)^{\le\alpha}$ is the \emph{full} subcomplex on the vertices of weight at most $\alpha$: it contains every cell of $\Delta(\pi)$ all of whose vertices have weight at most $\alpha$, with distinct cells and their attaching maps retained. The \emph{divisorial persistent homology} of $(X,\I)$ relative to $\pi$ is
\[
\DPH^\alpha_k(X,\I;\pi):=H_k\bigl(\Delta(\pi)^{\le\alpha};\Z\bigr),\qquad\alpha>0,
\]
with structure maps induced by $\Delta^{\le\alpha}\subseteq\Delta^{\le\beta}$ for $\alpha\le\beta$: a persistence module over $(\R_{>0},\le)$. For a field $\mathbf k$ we write $\DPH_k(\,\cdot\,;\mathbf k)$ for the same construction with $\mathbf k$-coefficients; statements about barcodes always refer to a field, since a persistence module of abelian groups need not decompose into interval summands (torsion). In the examples below the sublevel complexes are graphs and all structure maps are induced by inclusions of subgraphs, so $H_0$ and $H_1$ are read off from connected components and cycles, and the barcode is the same over every field.
\end{definition}

This is the dual complex convention of \cite{Ste06,dFKX17}; when all nonempty intersections $E_\sigma$ are connected it agrees with the usual simplicial nerve of $\{E_i\}$. The module changes only at the finitely many values of $\Gamma_\pi(X,\I)$, which are candidate critical values only: a vertex entering the filtration need not change the homology, and $\Delta^{\le\alpha}=\emptyset$ for $\alpha<\lambda_{\min}=\lct(\I)$; the first non-trivial group is $\DPH^{\lambda_{\min}}_0=\Z^c$, $c$ the number of components of the subcomplex spanned by the divisors computing the threshold. Vertices with $\nu_E=0$ (weight $+\infty$) never enter. Fullness of $\Delta^{\le\alpha}$ is used essentially below: the stellar subdivision of a full subcomplex at one of its simplices is the full subcomplex of the subdivision on the corresponding vertex set.

\subsection{Invariance under one blow-up}

Before the general statement we run the argument on the smallest example in which something happens.

\begin{example}\label{ex:threechain}
Let $\pi$ be a resolution of a surface germ whose boundary $D_\pi$ consists of three curves $E_1,E_2,E_3$ meeting in a chain, $E_1\cap E_2$ and $E_2\cap E_3$ single points, $E_1\cap E_3=\emptyset$, with weights $\lambda_1<\lambda_2<\lambda_3$ (all $\nu_j>0$). The dual complex is the path
\[
v_1\ \text{---}\ v_2\ \text{---}\ v_3 ,
\]
and the sublevel complexes are: $\emptyset$ for $\alpha<\lambda_1$; the single vertex $v_1$ for $\lambda_1\le\alpha<\lambda_2$; the edge $v_1v_2$ for $\lambda_2\le\alpha<\lambda_3$; the whole path for $\alpha\ge\lambda_3$. Every sublevel complex is contractible or empty: $\DPH_0=\Z$ born at $\lambda_1$, and nothing else.

Now blow up the point $Z=E_1\cap E_2$ (so $\sigma=\{1,2\}$, $r=s=2$). The new divisor $F$ has $\nu_F=\nu_1+\nu_2$, $a_F=a_1+a_2+1$, and by Lemma \ref{lem:mediant}
\[
\lambda_F=\frac{\nu_1\lambda_1+\nu_2\lambda_2}{\nu_1+\nu_2}\in[\lambda_1,\lambda_2].
\]
The strict transforms $E'_1,E'_2$ no longer meet; both meet $F$; $E'_3$ still meets $E'_2$. The new dual complex is the path
\[
v_1\ \text{---}\ v_F\ \text{---}\ v_2\ \text{---}\ v_3 ,
\]
the stellar subdivision of the old one at the edge $v_1v_2$. Compare the sublevel complexes level by level (Figure~\ref{fig:threechain}). For $\alpha<\lambda_1$ both are empty; for $\lambda_1\le\alpha<\lambda_F$ both are $\{v_1\}$; for $\lambda_F\le\alpha<\lambda_2$ the old complex is still $\{v_1\}$ while the new one is the edge $v_1v_F$, a cone over $v_1$; for $\lambda_2\le\alpha<\lambda_3$ the old complex is the edge $v_1v_2$ and the new one is the subdivided edge $v_1v_Fv_2$; for $\alpha\ge\lambda_3$ the whole paths. At every level the new complex is either the old one, or the old one with a cone attached along a vertex, or a subdivision of the old one: a homotopy equivalence, compatible with all inclusions.

Now suppose, counterfactually, that $\lambda_F$ were larger than $\lambda_2$. Then for $\lambda_2\le\alpha<\lambda_F$ the old sublevel complex would be the connected edge $v_1v_2$ while the new one would be the two isolated vertices $v_1,v_2$: the edge would have been \emph{lost} before its replacement $v_F$ arrived, and $\DPH_0$ would jump from $\Z$ to $\Z^2$ on the new model. The upper mediant bound $\lambda_F\le\max(\lambda_1,\lambda_2)$ is exactly what excludes this. Suppose instead that $\lambda_F$ were smaller than $\lambda_1$: then for $\lambda_F\le\alpha<\lambda_1$ the new complex would be the isolated vertex $v_F$ while the old one is empty, a class born too early. The lower bound $\lambda_F\ge\min(\lambda_1,\lambda_2)$ excludes this, and more generally guarantees (Theorem \ref{thm:blowup-inv}) that whenever $v_F$ enters a sublevel complex, it is attached along a cone with apex one of the $v_j$, $j\in\sigma$, that has already entered.

The same picture with the point $Z$ on $E_1$ only ($\sigma=\{1\}$, $r=2>s=1$) attaches a new leaf $v_F$ to $v_1$ with $\lambda_F=\lambda_1+1/\nu_1>\lambda_1$: at every level the new complex is the old one, possibly with a leaf attached to a vertex already present. Both patterns, subdivision and cone, are all that can happen in any dimension; this is the content of Lemma \ref{lem:dualblowup} and Theorem \ref{thm:blowup-inv}.
\end{example}

\begin{figure}[ht]
\centering\small
\renewcommand{\arraystretch}{1.5}
\begin{tabular}{@{}r@{\ \ }c@{\quad}c@{\quad}c@{\quad}c@{}}
 & $\lambda_1\!\le\!\alpha\!<\!\lambda_F$ & $\lambda_F\!\le\!\alpha\!<\!\lambda_2$ & $\lambda_2\!\le\!\alpha\!<\!\lambda_3$ & $\alpha\!\ge\!\lambda_3$\\
before: &
 \oldpath{vt}{vo}{vo}{eoff}{eoff} &
 \oldpath{vt}{vo}{vo}{eoff}{eoff} &
 \oldpath{vt}{vt}{vo}{eon}{eoff} &
 \oldpath{vt}{vt}{vt}{eon}{eon}\\
after: &
 \newpath{vt}{vo}{vo}{vo}{eoff}{eoff}{eoff} &
 \newpath{vt}{vt}{vo}{vo}{eon}{eoff}{eoff} &
 \newpath{vt}{vt}{vt}{vo}{eon}{eon}{eoff} &
 \newpath{vt}{vt}{vt}{vt}{eon}{eon}{eon}
\end{tabular}
\caption{Example~\ref{ex:threechain}: sublevel complexes of the chain before and after blowing up $E_1\cap E_2$ (grey: not yet present). At each level the new complex is the old one, the old one with a cone (the edge $v_1v_F$) attached along a vertex, or a subdivision of it.}\label{fig:threechain}
\end{figure}

Every example so far has a tree for a dual complex. The next one has a cycle, and it needs no blow-up and no singular $X$ at all: two smooth plane curves meeting at two points already produce it, on the identity resolution.

\begin{example}\label{ex:planecycle}
Let $X=\C^2$, $L=\{y=0\}$, $C=\{y=x^2-1\}$; $L$ and $C$ are smooth and meet transversally at $(\pm1,0)$, two points (at $x=\pm1$, $dC$ is not proportional to $dL=dy$). Let $\I=(f)$, $f=y^2(y-x^2+1)$, so $D_\pi=L\cup C$ is already simple normal crossings and $\pi=\mathrm{id}$ is a log-resolution: no blow-up is needed to produce a cycle, only two curves meeting twice. Here $\nu_L=2$, $\nu_C=1$, $a_L=a_C=0$ ($X$ smooth), so
\[
\lambda_L=\tfrac12,\qquad \lambda_C=1 .
\]
The stratum $L\cap C$ has \emph{two} connected components, the points $(1,0)$ and $(-1,0)$, so by Definition \ref{def:dual} the dual complex has one vertex for each curve and one edge for \emph{each} component of $L\cap C$: two parallel edges between $v_L$ and $v_C$, i.e.\ $\Delta(\pi)\simeq S^1$ (a bigon). The sublevel complexes are: $\emptyset$ for $\alpha<\tfrac12$; the single vertex $v_L$ for $\tfrac12\le\alpha<1$ ($v_C$ and both edges need $\lambda_C\le\alpha$ by fullness); and, for $\alpha\ge1$, both vertices and, since fullness only requires the two endpoints, both edges: the full bigon. Hence
\[
\DPH^\alpha_0=\Z\ (\alpha\ge\tfrac12),\qquad \DPH^\alpha_1=\begin{cases}0,&\alpha<1,\\ \Z,&\alpha\ge1,\end{cases}
\]
a single bar $[1,\infty)$ in degree $1$. The degree-one class is born at $\lambda_C=1$, strictly after the first critical value $\lambda_{\min}=\lambda_L=\tfrac12$: the divisor that closes the loop is the component with smaller vanishing order, $\nu_C=1<\nu_L=2$. Unlike the surface case of Section \ref{sec:surfaces}, no adaptedness is at issue here: $X$ is smooth, so every resolution is Jacobian-adapted (Definition \ref{def:adapted}), and the computation is exact as it stands.
\end{example}

\begin{figure}[ht]
\centering
\begin{tikzpicture}[x=1cm,y=1cm]
 \begin{scope}[xshift=0cm]
  \draw[->,black!40] (-2.2,0)--(2.4,0);
  \clip (-2.2,-1.3) rectangle (2.4,1.9);
  \draw[thick] (-2.2,0) -- (2.4,0);
  \draw[thick,dashed,black!70,domain=-1.7:1.7,samples=40,smooth] plot (\x,{\x*\x-1});
 \end{scope}
 \node[lab] at (2.0,0.2) {$L$}; \node[lab] at (0.3,0.9) {$C$};
 \fill (1,0) circle (1.6pt); \fill (-1,0) circle (1.6pt);
 \node[lab,above right] at (1,0) {$(1,0)$}; \node[lab,above left] at (-1,0) {$(-1,0)$};
 \node[lab] at (0,-1.9) {$D_\pi=L\cup C$};
 \begin{scope}[xshift=5cm]
  \node[vt,label={[lab]left:$v_L$}] (a) at (-1,0) {};
  \node[vt,label={[lab]right:$v_C$}] (b) at (1,0) {};
  \draw[eon] (a) to[bend left=55] (b);
  \draw[eon] (a) to[bend right=55] (b);
  \node[lab] at (0,0) {$\simeq S^1$};
  \node[lab] at (-1,-0.8) {$\lambda_L=\frac12$}; \node[lab] at (1,-0.8) {$\lambda_C=1$};
  \node[lab] at (0,-1.9) {$\Delta(\pi)$};
 \end{scope}
 \begin{scope}[xshift=8.4cm,yshift=0.2cm]
  \draw[->,black!60] (0,-0.9)--(3.2,-0.9) node[right,lab]{$\alpha$};
  \foreach \x/\t in {0.25/{1/2},1.0/{1}}{\draw[black!60] (\x*2,-0.95)--(\x*2,-0.85); \node[lab,below] at (\x*2,-0.95) {$\t$};}
  \draw[bar] (0.5,0.5)--(3.1,0.5) ;\node[lab,left] at (0.45,0.5) {$\mathrm{PH}_0$};
  \draw[bar] (2.0,-0.25)--(3.1,-0.25) ;\node[lab,left] at (1.95,-0.25) {$\mathrm{PH}_1$};
 \end{scope}
\end{tikzpicture}
\caption{Example~\ref{ex:planecycle}: the line $L$ and the parabola $C$ meet at two points, so the dual complex is a bigon; the degree-one bar is born at $\lambda_C=1$, after the first critical value $\lambda_L=\tfrac12$.}\label{fig:planecycle}
\end{figure}

Cells need to be tracked componentwise. Let $\rho\colon Y'\to Y$ be an snc blow-up as in Lemma \ref{lem:mediant}: the centre $Z$ is smooth and irreducible with simple normal crossings with $D_Y$, $\sigma=\{j:Z\subseteq E_j\}$, $s=|\sigma|\le r=\codim Z$, and $F$ is the exceptional divisor. Write $\Delta=\Delta(\pi)$ and $\Delta'=\Delta(\pi\circ\rho)$. A cell $C$ of $\Delta$ is a connected component of some $E_\omega$; by simple normal crossings $\omega=\omega(C)=\{i:C\subseteq E_i\}$ is its vertex set. In local coordinates adapted to $D_Y$ and $Z$ one has $Z=\{x_\sigma=0,\ y_1=\dots=y_{r-s}=0\}$ where the $y_\ell$ are not boundary coordinates; hence a cell $C$ with $C\subseteq Z$ exists only if $r=s$, and then $Z$ is itself a cell $\tau_Z$ of $\Delta$ (a component of $E_\sigma$) and the cells $C\subseteq Z$ are exactly the cells having $\tau_Z$ as a face (``cofaces of $\tau_Z$'').

\emph{Link data.} Let $\Lambda$ be the $\Delta$-complex whose cells are the pairs $(C,W)$, $C$ a cell of $\Delta$ and $W$ a connected component of $C\cap Z$, of dimension $|\omega(C)|-1$; the face of $(C,W)$ obtained by deleting $i\in\omega(C)$ is $(C_i,W_i)$, with $C_i$ the component of $E_{\omega(C)\setminus i}$ containing $C$ and $W_i$ the component of $C_i\cap Z$ containing $W$. Let $\kappa\colon\Lambda\to\Delta$, $(C,W)\mapsto C$ (it need not be injective: $C\cap Z$ may be disconnected). Let $\Delta_0\subseteq\Delta$ be the subcomplex of cells $C\not\subseteq Z$ and $\Lambda_0\subseteq\Lambda$ that of pairs $(C,W)$ with $C\not\subseteq Z$. If $r>s$, $\Delta_0=\Delta$ and $\Lambda_0=\Lambda$.

\begin{lemma}\label{lem:dualblowup}
In this situation:
\begin{enumerate}[label=\textup{(\alph*)},leftmargin=2em]
\item $\Delta'=\Delta_0\cup_\kappa(v_F*\Lambda_0)$: the cells of $\Delta'$ are the cells of $\Delta_0$ (strict transforms) and one cone cell $v_F*(C,W)$ for each $(C,W)\in\Lambda_0$ (the corresponding component of $F\cap E'_{\omega(C)}$), where we allow the empty cell $C=Y$, $\omega(C)=\emptyset$, $W=Z$, whose cone is the vertex $v_F$ itself; the face of $v_F*(C,W)$ opposite to $v_F$ is $C$, and the face opposite to $i\in\omega(C)$ is $v_F*(C_i,W_i)$.
\item If $r>s$, then $\Delta_0=\Delta$ and $\Delta'=\Delta\cup_\kappa(v_F*\Lambda)$: $\Delta$ is a subcomplex of $\Delta'$ and a cone on the link data is attached.
\item If $r=s$, then $\Delta_0$ is $\Delta$ minus the cofaces of $\tau_Z$, and $\Delta'$ is the stellar subdivision of $\Delta$ at $\tau_Z$: each coface $W$ of $\tau_Z$ is subdivided at the barycentre $v_F$ of $\tau_Z$, the simplices of the subdivision of $W$ being the faces of $W$ not containing $\tau_Z$ and the cones $v_F*D$ over them. The cone $v_F*D$ is the cell $v_F*(D,W)$ of (a) for the coface $W$ with $\omega(W)=\omega(D)\cup\sigma$; a cone over $D$ inside a larger coface is a face of the subdivision of that coface, and is the same cell.
\item For $(C,W)\in\Lambda$ and $j\in\sigma$ let $C^{(j)}$ be the component of $E_{\omega(C)\cup\{j\}}$ containing $W$. It is a cell of $\Delta$ with vertex set $\omega(C)\cup\{j\}$, $C^{(j)}=C$ if $j\in\omega(C)$, and $C$ is the face of $C^{(j)}$ opposite to $j$ if $j\notin\omega(C)$ (with the pair $(C^{(j)},W)\in\Lambda$ having $(C,W)$ as its face). If $r=s$ and $\sigma\not\subseteq\omega(C)\cup\{j\}$, then $C^{(j)}\not\subseteq Z$.
\end{enumerate}
\end{lemma}
\begin{proof}
\emph{Cells without $F$.} A component $C$ of $E_\omega$ with strict transform $C'$ is blown up along $C\cap Z$, a smooth closed subset of $C$, so $C'$ is the blow-up of the connected manifold $C$ along $C\cap Z$: connected, and empty exactly when $C\cap Z=C$, i.e.\ $C\subseteq Z$. This gives $\Delta_0$.

\emph{Cells with $F$.} In the local coordinates above $F\to Z$ is the projective bundle $\mathbb P(N_{Z/Y})$ with fibre $\mathbb P^{r-1}$ and homogeneous coordinates the normal coordinates $(x_\sigma,y)$. For $j\in\sigma$ the strict transform $E'_j$ meets the fibre in the hyperplane $\{x_j=0\}$; for $i\notin\sigma$ with $E_i\cap Z\neq\emptyset$, $E_i$ is transverse to $Z$ and $E'_i\cap F$ is the preimage of $E_i\cap Z$, a full fibre bundle. Hence over a point of $E_\omega\cap Z$ the fibre of $F\cap E'_\omega$ is the linear subspace of $\mathbb P^{r-1}$ cut out by the $|\omega\cap\sigma|$ hyperplanes $\{x_j=0\}$, $j\in\omega\cap\sigma$, which is a projective space $\mathbb P^{r-1-|\omega\cap\sigma|}$ if $|\omega\cap\sigma|<r$ and empty otherwise; and $|\omega\cap\sigma|=r$ iff $r=s$ and $\sigma\subseteq\omega$. A bundle with connected fibre over a connected base is connected, so the components of $F\cap E'_\omega$ correspond to the components $W$ of $E_\omega\cap Z$ for which this fibre is non-empty. A component $W$ lies in a unique component $C$ of $E_\omega$, and the condition fails iff $C\subseteq Z$. This gives (a); the face relations are induced by inclusions of strata.

(b) is (a) with $\Delta_0=\Delta$. For (c): the cells $(C,W)$ of $\Lambda_0$ are in bijection with the pairs $(D,W)$ with $W$ a coface of $\tau_Z$, $D$ a face of $W$ with $\sigma\not\subseteq\omega(D)$, and $\omega(W)=\omega(D)\cup\sigma$ (the last condition indexes each cone by the smallest coface containing it: for a triangle $abc$ subdivided at $\tau_Z=ab$ and $D=a$, only $(a,ab)$ qualifies, not $(a,abc)$, and there is a single edge $v_Fa$). Indeed, if $(C,W')\in\Lambda_0$, then $W'\subseteq E_{\omega(C)}\cap Z\subseteq E_{\omega(C)\cup\sigma}$ is a component of $E_{\omega(C)\cup\sigma}$ (dimensions agree by transversality), contained in $Z$, so it is a coface of $\tau_Z$ with vertex set $\omega(C)\cup\sigma$ having $C$ as a face; and conversely. The simplices of the stellar subdivision of the simplex $W$ at its face $\tau_Z$ are the faces of $W$ not containing $\tau_Z$ and the cones $v_F*D$ over them, which is the description (a) restricted to the cells inside $W$.

(d) $W\subseteq Z\subseteq E_j$ and $W\subseteq C\subseteq E_{\omega(C)}$, so the connected set $W$ lies in a component of $E_{\omega(C)\cup\{j\}}$, which is contained in the component $C$ of $E_{\omega(C)}$ containing $W$. The statements on faces follow from the definition of faces in $\Lambda$ (the component of $C\cap Z$ containing $W$ is $W$ itself). If $r=s$ and $\sigma\not\subseteq\omega(C)\cup\{j\}$ then $C^{(j)}\subseteq Z$ would force $\sigma\subseteq\omega(C^{(j)})=\omega(C)\cup\{j\}$.
\end{proof}

\begin{figure}[ht]
\centering
\begin{tikzpicture}[scale=1]
 \tikzset{tri/.style={fill=black!12,draw=none}}
 \begin{scope}
  \fill[tri] (0,0)--(1.6,0)--(0.8,1.2)--cycle;
  \draw[eon] (0,0)--(1.6,0)--(0.8,1.2)--cycle;
  \node[vt,label={[lab]left:$v_1$}] at (0,0) {};
  \node[vt,label={[lab]right:$v_2$}] at (1.6,0) {};
  \node[vt,label={[lab]above:$w$}] at (0.8,1.2) {};
  \draw[line width=1.6pt] (0,0)--(1.6,0); \node[lab,below] at (0.8,-0.05) {$\tau_Z=Z$};
  \draw[->,thick] (2.0,0.5)--(2.8,0.5);
  \begin{scope}[xshift=3.2cm]
   \fill[tri] (0,0)--(1.6,0)--(0.8,1.2)--cycle;
   \draw[eon] (0,0)--(1.6,0)--(0.8,1.2)--cycle;
   \draw[eon] (0.8,0)--(0.8,1.2);
   \node[vt,label={[lab]left:$v_1$}] at (0,0) {};
   \node[vt,label={[lab]right:$v_2$}] at (1.6,0) {};
   \node[vt,label={[lab]above:$w$}] at (0.8,1.2) {};
   \draw[line width=1.6pt] (0,0)--(1.6,0);
   \node[vt,label={[lab]below:$v_F$}] at (0.8,0) {};
  \end{scope}
  \node[lab] at (2.4,-0.9) {(a) $r=s$: stellar subdivision at $\tau_Z$};
 \end{scope}
 \begin{scope}[xshift=8.4cm]
  \draw[eon] (0,0)--(1.6,0);
  \node[vt,label={[lab]below:$v_1$}] at (0,0) {};
  \node[vt,label={[lab]below:$v_2$}] at (1.6,0) {};
  \draw[->,thick] (2.0,0)--(2.8,0);
  \begin{scope}[xshift=3.2cm]
   \draw[eon] (0,0)--(1.6,0);
   \draw[eon] (0,0)--(0.8,0.9);
   \node[vt,label={[lab]below:$v_1$}] at (0,0) {};
   \node[vt,label={[lab]below:$v_2$}] at (1.6,0) {};
   \node[vt,label={[lab]above:$v_F$}] at (0.8,0.9) {};
  \end{scope}
  \node[lab] at (2.4,-0.9) {(b) $r>s$ ($\sigma=\{1\}$): cone attached along $L$};
 \end{scope}
\end{tikzpicture}
\caption{The two effects of an snc blow-up on the dual complex (Lemma~\ref{lem:dualblowup}).}\label{fig:blowup}
\end{figure}

\begin{theorem}\label{thm:blowup-inv}
Let $\pi\colon Y\to X$ be a Jacobian-adapted log-resolution (if $X$ is singular; no hypothesis if $X$ is smooth) and $\rho\colon Y'\to Y$ an snc blow-up; by Lemma \ref{lem:adaptdom}(i), $\pi\circ\rho$ is again Jacobian-adapted. Write $\Delta=\Delta(\pi)$, $\Delta'=\Delta(\pi\circ\rho)$. There is a cellular map
\[
r\colon\ \bigcup_{\alpha}\Delta'^{\le\alpha}\longrightarrow\bigcup_\alpha\Delta^{\le\alpha}
\]
with $r(\Delta'^{\le\alpha})\subseteq\Delta^{\le\alpha}$ for every $\alpha>0$, such that each restriction $r_\alpha\colon\Delta'^{\le\alpha}\to\Delta^{\le\alpha}$ is a homotopy equivalence. Consequently $r$ induces an isomorphism of persistence modules $\DPH^\bullet_\ast(X,\I;\pi\circ\rho)\cong\DPH^\bullet_\ast(X,\I;\pi)$.
\end{theorem}
\begin{proof}
Naturality is built in: $r$ is one map on the union of the sublevel complexes, it maps $\Delta'^{\le\alpha}$ into $\Delta^{\le\alpha}$ for every $\alpha$, hence commutes with the inclusions for $\alpha\le\beta$; so it induces a morphism of persistence modules, an isomorphism as soon as every $r_\alpha$ is a homotopy equivalence. It remains to construct $r$ and to prove the latter. No identification of $\Delta'^{\le\alpha}$ with $\Delta^{\le\alpha}$ other than $r$ is used (the subdivision homeomorphism of Lemma \ref{lem:dualblowup}(c) is not a filtered map, because it sends $v_F$ to an interior point of $\tau_Z$).

\emph{The case $\nu_F=0$.} Then $\lambda_F=+\infty$, $v_F$ is in no sublevel complex, and a cell $W\subseteq Z$ of $\Delta$ has a vertex $j\in\sigma$ with $\nu_j=0$ (since $\nu_F=\sum_{j\in\sigma}\nu_j$), hence a vertex of weight $+\infty$. So $\bigcup_\alpha\Delta^{\le\alpha}\subseteq\Delta_0$, the two unions of sublevel complexes coincide with the cells of $\Delta_0$ of finite weight, and $r=\mathrm{id}$.

\emph{Construction of $r$ when $\nu_F>0$.} Let $j_0\in\sigma$ minimize $\lambda_j$ over $\sigma$; then $\nu_{j_0}>0$ and, by Lemma \ref{lem:mediant}(i), $\lambda_{j_0}\le\lambda_F$. Put $r=\mathrm{id}$ on the cells of $\Delta_0$ and $r(v_F*(C,W))=C^{(j_0)}$ for $(C,W)\in\Lambda_0$, the simplex $v_F*(C,W)$ being mapped linearly onto $C^{(j_0)}$ by the vertex map $v_F\mapsto v_{j_0}$, all other vertices fixed (a collapse if $j_0\in\omega(C)$). This is compatible with faces, hence defines a cellular map $\Delta'\to\Delta$ (on the cells of finite weight). We check it for the faces of $v_F*(C,W)$, $\omega=\omega(C)$, using Lemma \ref{lem:dualblowup}(d) and the fact that two components of the same $E_\eta$ that share a point coincide.
\begin{itemize}[leftmargin=2em,itemsep=1pt]
\item \emph{Face opposite $v_F$} is $C$, and $r(C)=C$. The linear map on $v_F*(C,W)$ sends the face with vertex set $\omega$ onto the face of $C^{(j_0)}$ with vertex set $\omega$, which is $C$ by (d) (equality if $j_0\in\omega$, opposite face to $j_0$ otherwise).
\item \emph{Face opposite $i\in\omega$, $i\ne j_0$} is $v_F*(C_i,W_i)$ with $r=(C_i)^{(j_0)}$, the component of $E_{(\omega\setminus i)\cup\{j_0\}}$ containing $W_i\supseteq W$. The image of this face under the linear map on $v_F*(C,W)$ is the face of $C^{(j_0)}$ with vertex set $(\omega\cup\{j_0\})\setminus i$, i.e.\ the component of the same $E_{(\omega\setminus i)\cup\{j_0\}}$ containing $C^{(j_0)}\supseteq W$. Both components contain $W$, hence coincide.
\item \emph{Face opposite $i=j_0\in\omega$} (here $r$ collapses $v_F*(C,W)$ onto $C$) is $v_F*(C_{j_0},W_{j_0})$, and $r$ of it is the component of $E_\omega$ containing $W_{j_0}$. Now $W_{j_0}$ is connected, contained in $E_{\omega\setminus j_0}\cap Z\subseteq E_\omega$ (because $Z\subseteq E_{j_0}$) and contains $W\subseteq C$; so it lies in the component $C$ of $E_\omega$, and $r(v_F*(C_{j_0},W_{j_0}))=C$. On the other hand the linear map on $v_F*(C,W)$ sends the vertices $(\omega\setminus j_0)\cup\{F\}$ of this face onto all of $\omega$, i.e.\ onto $C$. The two agree.
\end{itemize}
(If $j_0\in\omega$ and $i\ne j_0$ the second item reads: the face opposite $i$ of the simplex $C$, namely $C_i$, equals $(C_i)^{(j_0)}$ since $j_0\in\omega\setminus i$.) If a cell of $\Delta'^{\le\alpha}$ contains $v_F$ then $\lambda_{j_0}\le\lambda_F\le\alpha$ and the other vertices have weight $\le\alpha$, so its image has all vertices of weight $\le\alpha$: $r(\Delta'^{\le\alpha})\subseteq\Delta^{\le\alpha}$.

\emph{$r_\alpha$ is a homotopy equivalence.} Fix $\alpha>0$. There are three cases.

\emph{Case 1: $\lambda_F>\alpha$.} Then $\Delta'^{\le\alpha}=\Delta_0^{\le\alpha}$. If $r=s$ and a coface of $\tau_Z$ lay in $\Delta^{\le\alpha}$, then $\tau_Z\in\Delta^{\le\alpha}$, all $\nu_j>0$ and $\lambda_j\le\alpha$ for $j\in\sigma$, and the upper mediant bound, Lemma \ref{lem:mediant}(ii) (this is where adaptedness is used), would give $\lambda_F\le\max_{j\in\sigma}\lambda_j\le\alpha$, a contradiction. So $\Delta^{\le\alpha}\subseteq\Delta_0$, $\Delta'^{\le\alpha}=\Delta^{\le\alpha}$ and $r_\alpha=\mathrm{id}$. (If $r>s$ this holds without any bound, as $\Delta_0=\Delta$.)

\emph{Case 2: $\lambda_F\le\alpha$, and not both of $r=s$ and $\tau_Z\in\Delta^{\le\alpha}$.} Then $\Delta^{\le\alpha}\subseteq\Delta_0$ (a coface of $\tau_Z$ in $\Delta^{\le\alpha}$ would force $\tau_Z\in\Delta^{\le\alpha}$), so the inclusion $\iota\colon\Delta^{\le\alpha}\to\Delta'^{\le\alpha}$ exists and $r_\alpha\iota=\mathrm{id}$. We show $\iota r_\alpha\simeq\mathrm{id}$. The two maps agree on $\Delta_0^{\le\alpha}$. Let $v_F*(C,W)$ be a cell of $\Delta'^{\le\alpha}$; then $\omega(C)$ has weights $\le\alpha$. We claim $(C^{(j_0)},W)\in\Lambda_0$: if $r=s$, some $j_1\in\sigma$ has $\lambda_{j_1}>\alpha$ (as $\tau_Z\notin\Delta^{\le\alpha}$), so $j_1\notin\omega(C)\cup\{j_0\}$ and Lemma \ref{lem:dualblowup}(d) applies; if $r>s$ there is nothing to check. Hence $v_F*(C^{(j_0)},W)$ is a cell of $\Delta'$, with vertices of weight $\le\alpha$ (as $\lambda_{j_0}\le\lambda_F\le\alpha$), so it lies in $\Delta'^{\le\alpha}$, and it has $v_F*(C,W)$ as a face and $C^{(j_0)}=\iota r(v_F*(C,W))$ as the face opposite to $v_F$. Thus $\mathrm{id}$ and $\iota r_\alpha$ map each cell into a common closed cell. The straight-line homotopy $H(x,t)=(1-t)x+t\,r(x)$, with barycentric coordinates indexed by vertex labels ($r$ moves the coordinate of $v_F$ to $v_{j_0}$), is well defined on common faces for the same reason as in Case 3, and it stays in the closed cell $v_F*(C^{(j_0)},W)$, which contains $v_F*(C,W)$ as a face (Lemma \ref{lem:dualblowup}(d)). So $r_\alpha$ is a homotopy equivalence with inverse $\iota$. This case uses only the lower bound (i).

\emph{Case 3: $r=s$ and $\tau_Z\in\Delta^{\le\alpha}$.} Then all $\nu_j>0$ for $j\in\sigma$, and $\lambda_F\le\max_{j\in\sigma}\lambda_j\le\alpha$ by Lemma \ref{lem:mediant}(ii). By Lemma \ref{lem:dualblowup}(c), $\Delta'^{\le\alpha}$ is the stellar subdivision of $\Delta^{\le\alpha}$ at $\tau_Z$. Indeed, a cone cell $v_F*(D,W)$ of $\Delta'$ with all vertices of weight $\le\alpha$ has $\omega(W)=\omega(D)\cup\sigma$ of weights $\le\alpha$ (the weights of $\sigma$ are $\le\alpha$ because $\tau_Z\in\Delta^{\le\alpha}$), so the coface $W$ lies in $\Delta^{\le\alpha}$ and $v_F*(D,W)$ is one of the simplices of its subdivision; conversely every simplex of the subdivision of a coface $W\in\Delta^{\le\alpha}$ has vertices in $\omega(W)\cup\{F\}$, all of weight $\le\alpha$, and the cells of $\Delta_0^{\le\alpha}=\Delta^{\le\alpha}\cap\Delta_0$ are unchanged. Let $h\colon|\Delta'^{\le\alpha}|\to|\Delta^{\le\alpha}|$ be the subdivision homeomorphism, the identity of the underlying space, which sends $v_F$ to the barycentre of $\tau_Z$. For a cone cell $v_F*(D,W)$ the carrier of $h$ is the closed simplex $W$, and $r$ maps $v_F*(D,W)$ into the face $D^{(j_0)}$ of $W$ with vertex set $\omega(D)\cup\{j_0\}\subseteq\omega(W)$ (it is the component of $E_{\omega(D)\cup\{j_0\}}$ containing $W$, i.e.\ the face of $W$ with that vertex set). So $h(x)$ and $r_\alpha(x)$ both lie in the closed simplex $W$, and for cells of $\Delta_0$, $h=r_\alpha=\mathrm{id}$. Set $H(x,t)=(1-t)\,h(x)+t\,r_\alpha(x)$, computed in barycentric coordinates indexed by the vertex labels of $W$: if $x$ has coordinates $b_i$ on $\omega(D)$ and $b_F$ on $v_F$, then $h(x)$ has coordinates $b_i$ on $\omega(D)$ and $b_F/s$ added on each $j\in\sigma$, and $r_\alpha(x)$ has $b_F$ added on $j_0$. These formulas depend only on the coordinates and labels of $x$, hence agree on the common faces of two cone cells (whose carriers are then faces of both cofaces), so $H$ is a well-defined homotopy in $|\Delta^{\le\alpha}|$ from $h$ to $r_\alpha$. Thus $r_\alpha\simeq h$, a homeomorphism, and $r_\alpha$ is a homotopy equivalence.

In every case $r_\alpha$ is a homotopy equivalence, and the three cases cover all $\alpha$: if $\lambda_F\le\alpha$ we are in Case 2 or 3, and otherwise in Case 1.
\end{proof}

\begin{remark}\label{rem:mediantrole}
For a general snc centre on a Jacobian-adapted log-resolution (a smooth irreducible centre with simple normal crossings with the boundary, whether or not it is a stratum of $D_\pi$), the lower mediant bound of Lemma \ref{lem:mediant}(i) is unconditional and suffices for the case $r>s$ of the proof. In the stratum case $r=s$ the proof also needs the upper bound of Lemma \ref{lem:mediant}(ii), which uses adaptedness and is available when every boundary component through the centre has $\nu_j>0$; the proof uses it only to rule out $\lambda_F>\alpha$ when $\tau_Z\in\Delta^{\le\alpha}$ (Case 1), and otherwise needs only the lower bound, which provides the filtered map $r$. The upper bound $\lambda_F\le\max$ prevents a simplex $\tau_Z$ from leaving a sublevel complex before its replacement $v_F$ arrives; the lower bound $\lambda_F\ge\min$ guarantees that $v_F\mapsto v_{j_0}$ is a filtered map, i.e.\ that a newly created vertex is attached along a cone whose apex has already entered. Both are the divisorial analogue of the fact that blowing up cannot decrease log canonical thresholds.
\end{remark}

\begin{example}\label{ex:A2fail}
Let $X=\{xy=z^3\}\subset\C^3$, $\I=\mathfrak m$, and $\pi\colon Y\to X$ the minimal resolution, with exceptional curves $E_1,E_2$ meeting at one point $p$ (Example \ref{ex:An} with $n=2$). Near $p$ there are coordinates $(u,v)$ with $E_1=\{u=0\}$, $E_2=\{v=0\}$ and
\[
x=uv^2,\qquad y=u^2v,\qquad z=uv .
\]
By Remark \ref{rem:mather}, $\Jac_{Y/X}=j_X\mathcal O_Y$ (as $K_{Y/X}=0$) with $j_X=(y,x,3z^2)$, so near $p$
\[
\Jac_{Y/X}=(u^2v,\ uv^2,\ u^2v^2)=uv\cdot(u,v):
\]
$a_1=a_2=1$, and $\ord_p(\Jac_{Y/X})=3=a_1+a_2+1$, so $e_p=1$: the ideal is not invertible at $p$, and $\pi$ is not Jacobian-adapted (indeed $j_X\mathcal O_Y$ is not principal at $p$). With $\nu_1=\nu_2=1$ the weights are $\lambda_1=\lambda_2=2$. Blow up $p$; for the new curve $F$, $\nu_F=2$ and $a_F=a_1+a_2+1+e_p=4$, so
\[
\lambda_F=\frac{4+1}{2}=\frac52>\max(\lambda_1,\lambda_2)=2 ,
\]
against the value $2$ that the exact formula would predict, and the upper mediant bound fails. Directly: $\ord_F$ is the order of vanishing at $p$, so $\ord_F(x)=\ord_F(y)=3$, $\ord_F(z)=2$, $\ord_F(j_X)=3$, $k_F=1$, and $a_F=k_F+\ord_F(j_X)=4$. On $Y'$ the ideal $\Jac_{Y'/X}$ is invertible (in the chart $u=u_1$, $v=u_1v_1$, where $F=\{u_1=0\}$ and $E'_2=\{v_1=0\}$, it is $(u_1)\cdot\rho^*(uv(u,v))=(u_1^4v_1)$, and symmetrically in the other chart) and $\pi\circ\rho$ is Jacobian-adapted. Three consequences follow. The dual graph of $Y$ is the edge $v_1v_2$ with both weights $2$, that of $Y'$ is the path $v_1v_Fv_2$ with weights $(2,\tfrac52,2)$; for $2\le\alpha<\tfrac52$ the first sublevel graph is connected and the second is not, so the filtered dual complexes of the two log-resolutions have different homology and Theorem \ref{thm:blowup-inv} cannot hold for non-adapted log-resolutions. The count of Proposition \ref{prop:mp} gives $2$ on $Y$ and $1$ on $Y'$; the true multiplicity is $m(0)=1$, since $\mathcal H^4(X\cap B_{\sqrt\varepsilon})\asymp\varepsilon^2$ with no logarithm for the germ of a complex surface at a point, and it is $Y'$ that computes it. And on $Y$ the density $\pi^*\mathcal H^4$ near $p$ is $\asymp|uv|^2(|u|^2+|v|^2)$, not $|uv|^2$: the extra factor is what removes the logarithm from $\int_{|uv|^2\le\varepsilon}\pi^*\mathcal H^4$, in accordance with Theorem \ref{thm:volume} being stated for adapted resolutions only.
\end{example}

\subsection{Independence of the resolution: the factorization hypothesis}

Theorem \ref{thm:blowup-inv} compares $\pi$ with $\pi\circ\rho$, both Jacobian-adapted. To compare two arbitrary adapted log-resolutions $\pi_1\colon Y_1\to X$ and $\pi_2\colon Y_2\to X$ by this method one needs a chain
\[
Y_1=W_0\ \dashrightarrow\ W_1\ \dashrightarrow\ \cdots\ \dashrightarrow\ W_N=Y_2
\]
of models over $X$ in which each arrow is an snc blow-up or the inverse of one, and every $W_i$ is Jacobian-adapted. Proposition \ref{prop:domination} does not give this: it gives $W\to Y_1$ by snc blow-ups (adapted by Lemma \ref{lem:adaptdom}(i), since $Y_1$ is adapted) and $W\to Y_2$ merely a morphism, in general not a sequence of snc blow-ups, let alone one through adapted log-resolutions. The required input is the following hypothesis, \emph{weak factorization by snc blow-ups}, abbreviated \textup{(WF)}.

\begin{hypothesis}\label{hyp:WF}
Any two Jacobian-adapted log-resolutions of $(X,\I)$ are connected by a finite chain of snc blow-ups and inverses of snc blow-ups, through models that are Jacobian-adapted log-resolutions of $(X,\I)$ (in particular, over $X$).
\end{hypothesis}

\begin{proposition}\label{prop:WFalg}
Let $X$ be a reduced complex algebraic variety, $\I\subset\mathcal O_X$ a nonzero coherent ideal, and let $\pi_1\colon Y_1\to X$, $\pi_2\colon Y_2\to X$ be Jacobian-adapted log-resolutions of $(X,\I)$ in the algebraic category. Then $Y_1,Y_2$ are connected by a finite chain of snc blow-ups and inverses of snc blow-ups through Jacobian-adapted log-resolutions of $(X,\I)$. In other words \textup{(WF)} holds for $(X,\I)$ and Theorem~\ref{thm:invariance} is unconditional for algebraic pairs.
\end{proposition}
\begin{proof}
Lemma~\ref{lem:adaptdom}(i) and Definition~\ref{def:sncblowup} are used in their algebraic form; the proofs are the chain rule and the local description of centres.

\emph{Step 1: a common boundary.} Put $T_i=\pi_i(D_{\pi_i})$ and $T=T_1\cup T_2$, a proper closed subset of $X$. Since $Y_i$ is smooth, $\pi_i$ is not a local isomorphism at any point over $\operatorname{Sing}X$, so $\operatorname{Sing}X\subseteq T$; also $V(\I)\subseteq T$. We claim that every prime divisor $C\subset\pi_i^{-1}(T)$ lies in $D_{\pi_i}$. Otherwise $\pi_i$ is a local isomorphism at the generic point of $C$ and $\I\mathcal O_{Y_i}$ is trivial there, so $\pi_i(C)$ is a codimension-one subset of $X_{\mathrm{reg}}\setminus V(\I)$ contained in $T_1\cup T_2$, hence in $\pi_j(\Exc\pi_j)$ for some $j$. But $\pi_j$ is a proper birational morphism onto the normal variety $X_{\mathrm{reg}}$ (over it), hence an isomorphism over all codimension-one points, so $\pi_j(\Exc\pi_j)\cap X_{\mathrm{reg}}$ has codimension $\ge2$: contradiction. By Bierstone--Milman \cite{BM97} applied to the ideal of $T$ on $(Y_i,D_{\pi_i})$ there is a finite sequence of snc blow-ups $\rho_i\colon Y_i'\to Y_i$ with centres in $\pi_i^{-1}(T)$ such that $\pi_i'^{-1}(T)$ is the support of an invertible ideal, i.e.\ a divisor. Its components are $\rho_i$-exceptional or strict transforms of divisorial components of $\pi_i^{-1}(T)$, which lie in $D_{\pi_i}$ by the claim; hence they lie in $D_{\pi_i'}=\rho_i^{-1}(D_{\pi_i})\cup\{\text{exceptional divisors}\}$. Conversely $D_{\pi_i'}\subseteq\pi_i'^{-1}(T)$. Thus $D_{\pi_i'}=\pi_i'^{-1}(T)$, and $Y_i'$ is a Jacobian-adapted log-resolution by Lemma~\ref{lem:adaptdom}(i). It suffices to connect $Y_1'$ to $Y_2'$: the full chain is $Y_1\to Y_1'\dashrightarrow Y_2'\to Y_2$ read with the inverse blow-ups $\rho_2^{-1}$.

\emph{Step 2: a common open set.} Let $U_X=X\setminus T\subseteq X_{\mathrm{reg}}$ and $U_i=\pi_i'^{-1}(U_X)=Y_i'\setminus D_{\pi_i'}$. Then $\pi_i'|_{U_i}\colon U_i\to U_X$ is proper, birational, \'etale (no point of $U_i$ is exceptional) and finite, hence an isomorphism. So the birational map $\phi=\pi_2'^{-1}\circ\pi_1'\colon Y_1'\dashrightarrow Y_2'$ is an isomorphism $U:=U_1\to U_2$, and $Y_1'\setminus U=D_{\pi_1'}$, $Y_2'\setminus\phi(U)=D_{\pi_2'}$ are simple normal crossings divisors.

\emph{Step 3: $X$ complete.} Then $Y_1',Y_2'$ are complete and nonsingular, and \cite[Thm.~0.3.1]{AKMW02} applied to $\phi$ and $U$ gives $Y_1'=V_0\dashrightarrow\cdots\dashrightarrow V_N=Y_2'$ and an index $i_0$ such that: each step is a blow-up with nonsingular centre disjoint from $U$ or the inverse of one (3); $V_i\to Y_1'$ is a morphism for $i\le i_0$ and $V_i\to Y_2'$ is a morphism for $i\ge i_0$ (5); and the centres have normal crossings with the inverse image of $Y_j'\setminus U$ (6). Call $\psi_i\colon V_i\to Y_j'$ this morphism. Over $U$ it is an isomorphism: the copy of $U\subset V_i$ maps isomorphically onto $U$, and the fibres of the proper birational morphism $\psi_i$ onto the smooth variety $Y_j'$ are connected (Zariski), so the fibre over a point of $U$ is that single point. Hence $V_i\setminus U=\psi_i^{-1}(D_{\pi_j'})$ is a simple normal crossings divisor. We claim that $V_i\setminus U=D_{V_i}:=\Exc(V_i\to X)\cup\Supp(\I\mathcal O_{V_i})$. The inclusion $\supseteq$ holds since $V_i\to X$ is an isomorphism over $U_X$. For $\subseteq$ let $y\in\psi_i^{-1}(D_{\pi_j'})$ and $z=\psi_i(y)$. If $z\in\Supp(\I\mathcal O_{Y_j'})$ then $y\in\Supp(\I\mathcal O_{V_i})$. If $z\in\Exc(\pi_j')$ and $\pi_j'\psi_i$ were a local isomorphism at $y$, then $\psi_i$, a proper birational morphism of manifolds, would be a local isomorphism at $y$, and so would $\pi_j'$ at $z$: contradiction. Thus every $V_i$ is a log-resolution of $(X,\I)$ (note $\I\mathcal O_{V_i}=\psi_i^{-1}(\I\mathcal O_{Y_j'})\mathcal O_{V_i}$ is invertible), and it is Jacobian-adapted by Lemma~\ref{lem:adaptdom}(i). A blow-up with a disconnected nonsingular centre is a composition of blow-ups of its (pairwise disjoint) components, and each is an snc blow-up of a model with boundary $V_i\setminus U$, by (6). This is the required chain.

\emph{Step 4: $X$ arbitrary.} Choose a compactification $X\subset\overline X$ (Nagata). The closure of the graph of $\pi_i'$ in a compactification of $Y_i'$ over $\overline X$, resolved by Hironaka's theorem and followed by a principalization of the ideal of the complement of $U$ with centres in the locus where that complement is not already a simple normal crossings divisor, produces a complete nonsingular variety $\overline Y_i\supset Y_i'$, a morphism $\overline Y_i\to\overline X$ whose restriction over $X$ is $\pi_i'$ (so the inverse image of $X$ is $Y_i'$, because $\pi_i'$ is proper), and a simple normal crossings divisor $\overline Y_i\setminus U$. Now apply Step~3 to $\overline Y_1\dashrightarrow\overline Y_2$ and restrict the resulting chain to the inverse images of $X$ (each $\overline V_i\to\overline X$ is defined through $\overline Y_j$): a blow-up restricts to the blow-up of the part of the centre lying over $X$, which is trivial if empty and is smooth, disjoint from $U$ and in normal crossing position otherwise. The restricted chain joins $Y_1'$ to $Y_2'$ and satisfies the conclusions of Step~3 verbatim.
\end{proof}

\begin{remark}[Non-algebraic analytic pairs]\label{rem:WFreal}
The weak factorization theorem \cite[Thm.~0.3.1]{AKMW02} is a theorem on algebraic varieties over a field of characteristic zero, and Proposition~\ref{prop:WFalg} uses it in that form, together with a compactification (Step~4) that is available only for algebraic $X$. For germs of complex analytic spaces that are not algebraic we use no analytic version of it, and \textup{(WF)} for such pairs in dimension $\ge3$ remains a hypothesis; for surface germs it is Lemma~\ref{lem:surfdom}.
\end{remark}

\begin{remark}\label{rem:WFstatus}
When $X$ is smooth every log-resolution is adapted and \textup{(WF)} is the unrestricted statement.
\begin{enumerate}[label=\textup{(\alph*)},leftmargin=2em]
\item In the algebraic setting ($X$ an algebraic variety, $\I$ an ideal sheaf), \textup{(WF)} is Proposition \ref{prop:WFalg}, whose proof is the weak factorization theorem of Abramovich--Karu--Matsuki--W\l odarczyk \cite{AKMW02}, in the relative form used for the unweighted dual complex in \cite{dFKX17}. Two points are specific to the weighted setting and are the content of Steps 1 and 3 of the proof: the two boundaries are first put in common position, and the adaptedness of every intermediate model is automatic, since each of them dominates $Y_1$ or $Y_2$ by a morphism of manifolds (Lemma \ref{lem:adaptdom}(i)); no modification of the factorization algorithm is needed. The weak factorization theorem is used here for complete varieties, and the passage from complete $X$ to arbitrary $X$ is Step 4.
\item For germs of complex analytic surfaces, \textup{(WF)} holds unconditionally (Lemma \ref{lem:surfdom}): any two adapted log-resolutions are dominated by a third obtained from each by point blow-ups, and every model in the two sequences dominates $Y_1$ or $Y_2$, hence is adapted. Example \ref{ex:A2fail} is the converse phenomenon: it exhibits a point blow-up that \emph{enters} the adapted class from a non-adapted log-resolution, which is why both endpoints of the chain must be adapted.
\item For complex analytic pairs that are not algebraic, in dimension $\ge3$, \textup{(WF)} is a hypothesis (Remark \ref{rem:WFreal}). The definitions of Section \ref{sec:dual} make sense on adapted log-resolutions without it; only the independence of $\DPH$ from the adapted resolution, Theorem \ref{thm:invariance}, depends on it. A model-free proof, via the valuation-space skeleta of \cite{MN15,BFJ08}, would remove the hypothesis.
\end{enumerate}
\end{remark}

\begin{theorem}\label{thm:invariance}
Assume \textup{(WF)} for $(X,\I)$. Then for any two Jacobian-adapted log-resolutions $\pi,\pi'$ and every $\alpha>0$ the sublevel complexes $\Delta(\pi)^{\le\alpha}$ and $\Delta(\pi')^{\le\alpha}$ are homotopy equivalent, by equivalences compatible with the inclusions $\Delta^{\le\alpha}\subseteq\Delta^{\le\beta}$ (a zigzag of the filtered maps of Theorem \ref{thm:blowup-inv} and their inverses on homology). Consequently $\DPH^\alpha_\bullet(X,\I)$ is, up to isomorphism of persistence modules, independent of the adapted log-resolution.
\end{theorem}
\begin{proof}
Each snc blow-up $W_{i+1}\to W_i$ in the chain provided by (WF) induces an isomorphism of persistence modules $\DPH(W_{i+1})\to\DPH(W_i)$ by Theorem \ref{thm:blowup-inv}, applicable because every model in the chain is Jacobian-adapted; for an inverse blow-up use the inverse isomorphism. Compose the resulting finite chain of isomorphisms.
\end{proof}

\begin{corollary}\label{cor:WFalg}
For a complex algebraic pair $(X,\I)$, $X$ possibly singular, $\DPH^\alpha_\bullet(X,\I)$ is independent, as a persistence module, of the Jacobian-adapted algebraic log-resolution.
\end{corollary}
\begin{proof}
Proposition \ref{prop:WFalg} gives \textup{(WF)}; apply Theorem \ref{thm:invariance}.
\end{proof}

\begin{remark}\label{rem:status}
The local threshold (Section \ref{sec:threshold}) needs only Proposition \ref{prop:domination}, no adaptedness, and holds for complex analytic pairs in every dimension. For the filtered dual complex, always restricted to Jacobian-adapted log-resolutions when $X$ is singular:
\begin{itemize}[leftmargin=2em,itemsep=1pt]
\item[--] \emph{all dimensions:} invariance under one snc blow-up of an adapted log-resolution (Theorem \ref{thm:blowup-inv}), unconditional;
\item[--] \emph{complex algebraic pairs, any dimension, $X$ possibly singular:} independence of the algebraic adapted resolution, unconditional (Proposition \ref{prop:WFalg}, Corollary \ref{cor:WFalg});
\item[--] \emph{surface germs, complex analytic, normal $X$ with $\mathfrak m_0$-primary $\I$:} full independence of the resolution among adapted resolutions (Theorem \ref{thm:surface}), unconditional, through \textup{(WF)} verified directly in Lemma \ref{lem:surfdom};
\item[--] \emph{general pairs (non-algebraic complex analytic, $n\ge3$):} independence of the resolution conditional on \textup{(WF)} (Theorem \ref{thm:invariance}).
\end{itemize}
Theorem \ref{thm:blowup-inv} compares an adapted log-resolution with its adapted blow-ups; the passage to arbitrary pairs of adapted log-resolutions is Hypothesis \ref{hyp:WF}. Non-adapted resolutions are excluded from Section \ref{sec:dual} on (Example \ref{ex:A2fail}).
\end{remark}

\begin{remark}\label{rem:excnotinv}
Replacing $\Delta^{\le\alpha}$ by the space $\Exc^{\le\alpha}=\bigcup_{\lambda_E\le\alpha}E$ does not work: its homology in low degrees agrees with that of $\Delta^{\le\alpha}$ only under strong hypotheses (rational components, degrees below $\dim_\R E$), and it is not birationally invariant: blowing up a point on a curve $E$ of a surface resolution adds a component $\mathbb P^1\simeq S^2$ attached at a point, changing $H_2$ while $\Delta^{\le\alpha}$ acquires a contractible leaf.
\end{remark}

\begin{remark}\label{rem:literature}
For $X$ smooth and complex algebraic, the function $\lambda$ is the restriction to divisorial valuations of $A(v)/v(\I)$ on the valuation space of \cite{JM12}, with $A$ the usual log discrepancy; in this smooth setting $\Delta(\pi)$ is the dual complex of the snc boundary, which is a skeleton in the Berkovich analytification, and the sublevel complexes $\Delta(\pi)^{\le\alpha}$ are comparable to intersections of the skeleton with sublevel sets of a weight function in the sense of \cite{MN15}; the invariance of their homotopy type under change of model is of the same nature as the invariance of the essential skeleton \cite{MN15,KX16,Pay13,NX16}. This is context only: none of the proofs uses valuation spaces or skeleta, and for singular $X$ (where $A_E=a_E+1$ involves the Mather discrepancy) we make no identification with the valuation space of \cite{JM12} (see Proposition~\ref{prop:arc} for the arc-space description). A comparison with \cite{NX16} and \cite{dFKX17}, including what they prove and what is specific to the present paper, is given in the Introduction (Relation to previous work). The identification of the weight, with the Mather discrepancy, through Theorem \ref{thm:volume} for the $2n$-dimensional measure, the Hodge-theoretic reading of Proposition~\ref{prop:hodge}, and the filtered, persistent form of the blow-up analysis (Theorem \ref{thm:blowup-inv}) are specific to the present paper.
\end{remark}

\subsection{Weight-zero Hodge theory of the sublevel unions}\label{subsec:hodge}
Let $\pi$ be a log-resolution and, for $\alpha>0$, let
\[
D^{\le\alpha}=\bigcup_{E\subset D_\pi,\ \lambda_E\le\alpha}E\ \subset\ \Xt
\]
be the union of the boundary divisors of weight at most $\alpha$, a reduced simple normal crossings variety. If $\nu_E>0$ then $\pi(E)\subseteq V(\I)$ and $E\subseteq\pi^{-1}(V(\I))$, which is complete when $V(\I)$ is complete and $\pi$ is proper (for instance $X$ complete, or a representative of a germ with $V(\I)$ a point); divisors with $\nu_E=0$ have weight $+\infty$ and never enter. So, in that case, every $D^{\le\alpha}$ is complete.

\begin{proposition}\label{prop:hodge}
Let $(X,\I)$ be a complex algebraic pair, $\pi$ an algebraic log-resolution, Jacobian-adapted if $X$ is singular, and assume that $V(\I)$ is complete. For every $k\ge0$ and $\alpha>0$ there is a natural isomorphism
\[
H^k\bigl(\Delta(\pi)^{\le\alpha};\Q\bigr)\ \xrightarrow{\ \cong\ }\ \operatorname{gr}^W_0H^k\bigl(D^{\le\alpha};\Q\bigr),
\]
compatible with the restriction maps induced by $D^{\le\alpha}\subseteq D^{\le\beta}$ and $\Delta^{\le\alpha}\subseteq\Delta^{\le\beta}$ for $\alpha\le\beta$. Consequently the persistence module $\DPH_k(X,\I;\Q)$ is the dual of the persistence module of weight-zero parts $\bigl(\operatorname{gr}^W_0H^k(D^{\le\alpha};\Q)\bigr)_\alpha$ of the mixed Hodge structures of Deligne. The same holds for a normal surface germ with $\mathfrak m_0$-primary $\I$ and the exceptional curves of weight at most $\alpha$.
\end{proposition}
\begin{proof}
Let $D=\bigcup_{i\in S}D_i$ be a complete simple normal crossings variety, $D^{(p)}=\bigsqcup_{|\sigma|=p+1}D_\sigma$ the disjoint union of the normalized strata (for $\sigma\subseteq S$ ordered), and consider the Mayer--Vietoris spectral sequence $E_1^{p,q}=H^q(D^{(p)};\Q)\Rightarrow H^{p+q}(D;\Q)$, whose $d_1$ is the alternating sum of restrictions. By Deligne \cite{Del71,Del74} it is the weight spectral sequence of the mixed Hodge structure on $H^\bullet(D;\Q)$: $E_1^{p,q}$ is pure of weight $q$, and the sequence degenerates at $E_2$, so $\operatorname{gr}^W_qH^k(D;\Q)=E_2^{k-q,q}$. For $q=0$ the row $E_1^{\bullet,0}=H^0(D^{(\bullet)};\Q)$ has one copy of $\Q$ for each connected component of each $D_\sigma$, with the alternating restriction differential: it is the cellular cochain complex of the dual complex $\Delta(D)$ of Definition~\ref{def:dual}, so $\operatorname{gr}^W_0H^k(D;\Q)=H^k(\Delta(D);\Q)$ (compare \cite{Pay13}). Apply this to $D=D^{\le\alpha}$: its components are the $E$ with $\lambda_E\le\alpha$ and its strata are the connected components of the intersections of such $E$, so $\Delta(D^{\le\alpha})$ is exactly the full subcomplex $\Delta(\pi)^{\le\alpha}$. The restriction maps along $D^{\le\alpha}\subseteq D^{\le\beta}$ are morphisms of mixed Hodge structures, hence strict for the weight filtration, and they are compatible with the Mayer--Vietoris spectral sequences, which gives the compatibility with $\Delta^{\le\alpha}\subseteq\Delta^{\le\beta}$. The statement on $\DPH$ follows from the universal coefficient theorem over the field $\Q$. For a surface germ the divisors are compact Riemann surfaces, hence complete algebraic curves, and the same argument applies.
\end{proof}

The rational coefficients are needed; with $\Z$ the statement is not claimed. The proposition places $\DPH$ inside a larger structure without enlarging it: the higher weights $\operatorname{gr}^W_qH^k(D^{\le\alpha})$, $q>0$, are \emph{not} controlled by $\DPH$ and are not invariants of the pair, since $H^\bullet(\Exc^{\le\alpha})$ itself changes under blow-ups (Remark~\ref{rem:excnotinv}); only the weight-zero part is invariant (Theorem~\ref{thm:blowup-inv}).

\subsection{The hierarchy of the invariants}\label{subsec:hierarchy}
Four objects are attached to a pair $(X,\I)$, and they sit at different levels; properties of one must not be attributed to another. Fix $p\in V(\I)$ and a Jacobian-adapted log-resolution $\pi$.
\begin{enumerate}[label=\textup{(\arabic*)},leftmargin=2em]
\item \emph{The divisorial spectrum} $\Gamma_\pi(X,\I)=\{\lambda_E:E\subset D_\pi\}$, a finite set of numbers. It depends on $\pi$: every snc blow-up adds weights (Lemma~\ref{lem:mediant}). Only its minimum over the divisors through $p$ is intrinsic (Theorem~\ref{thm:divcomp}).
\item \emph{The filtered dual complex} $\bigl(\Delta(\pi)^{\le\alpha}\bigr)_{\alpha}$, the weighted incidence structure. It is the complete combinatorial datum of $\pi$, it depends on $\pi$, and its filtered homotopy type is independent of the adapted $\pi$ in the settings of Table~\ref{tab:scope} (Theorems~\ref{thm:blowup-inv} and~\ref{thm:invariance}).
\item \emph{The persistence module} $\DPH_k(X,\I)=\bigl(H_k(\Delta^{\le\alpha})\bigr)_\alpha$, with its barcode and its effective critical values. It is a functor of (2), so it is invariant wherever (2) is, but it forgets the filtered homotopy type beyond homology.
\item \emph{The asymptotic pair} $(\lambda_{\min}(p),m(p))$, read from the volume (Theorem~\ref{thm:volume}). It is determined by (2): $\lambda_{\min}(p)$ is the lowest weight over $p$ and $m(p)-1$ is the largest dimension of a simplex at that weight whose stratum meets $\pi^{-1}(p)$ (Proposition~\ref{prop:mp}). It is \emph{not} determined by (3): $m$ is lost (Remark~\ref{rem:forgets}).
\end{enumerate}
Figure~\ref{fig:hierarchy} shows the arrows and Table~\ref{tab:hierarchy} the properties.

\begin{figure}[ht]
\centering
\begin{tikzpicture}[x=1cm,y=1cm,
  bx/.style={draw,rounded corners=2pt,thick,align=center,font=\scriptsize,inner sep=4pt,fill=black!5,minimum width=4.6cm},
  ar/.style={->,thick,>=Stealth}]
 \node[bx] (res) at (0,4.6) {resolution data of $\pi$\\ $\{(E,\nu_E,a_E)\}$ and incidences};
 \node[bx] (fc)  at (0,3.0) {(2) filtered dual complex $(\Delta(\pi)^{\le\alpha})_\alpha$\\ \emph{model dependent}};
 \node[bx] (ht)  at (0,1.4) {filtered homotopy type\\ \emph{invariant} (unconditional or under (WF))};
 \node[bx] (pm)  at (0,-0.2) {(3) module $\DPH_\bullet$, barcode\\ \emph{invariant}};
 \node[bx,fill=black!14] (pair) at (7.0,3.0) {(4) $(\lambda_{\min}(p),m(p))$\\ local, from the lowest level over $p$};
 \node[bx] (vol) at (7.0,4.6) {volume $\Hn(B\cap U_\varepsilon)\asymp\varepsilon^{\lambda_{\min}}\LOG^{m-1}$};
 \node[bx] (gam) at (7.0,-0.2) {first birth of $\DPH_0$ $=\lambda_{\min}$\\ (germ with all divisors over $p$)};
 \draw[ar] (res)--(fc);
 \draw[ar] (fc)--(ht) node[midway,right,font=\scriptsize]{adapted + (WF)};
 \draw[ar] (ht)--(pm) node[midway,right,font=\scriptsize]{homology};
 \draw[ar] (fc)--(pair) node[midway,above,font=\scriptsize]{Prop.~\ref{prop:mp}};
 \draw[ar] (pair)--(vol) node[midway,right,font=\scriptsize]{Thm.~\ref{thm:volume}};
 \draw[ar] (pm)--(gam) node[midway,above,font=\scriptsize]{only $\lambda_{\min}$};
 \draw[ar,dashed] (pm.north east)--(pair.south west) node[pos=0.5,below,sloped,font=\scriptsize]{$m$ lost};
\end{tikzpicture}
\caption{The hierarchy of invariants. Solid arrows are functorial maps (determination); the dashed arrow from $\DPH$ to $(\lambda_{\min},m)$ does not exist, because $m$ is not a function of the module (Remark~\ref{rem:forgets}). The volume asymptotics is attached to the asymptotic pair only.}
\label{fig:hierarchy}
\end{figure}

\begin{table}[ht]
\centering\footnotesize\renewcommand{\arraystretch}{1.3}
\begin{tabularx}{\textwidth}{@{}>{\raggedright\arraybackslash}p{3.6cm}>{\raggedright\arraybackslash}X>{\raggedright\arraybackslash}X>{\raggedright\arraybackslash}X>{\raggedright\arraybackslash}X@{}}
\toprule
Property & Spectrum $\Gamma_\pi$ & Filtered complex & Module $\DPH$ & Pair $(\lambda_{\min},m)$\\
\midrule
Independent of the (adapted) resolution & no (only its minimum over $p$) & up to filtered homotopy, in the settings of Table~\ref{tab:scope} & yes, in the same settings & yes (Cor.~\ref{cor:mp}, Thm.~\ref{thm:divcomp})\\
Read from the volume asymptotics & minimum only & no & only the first birth, $\lambda_{\min}$ & yes\\
Sees incidences and cycles & no & yes & yes, through homology & only the top simplex at the minimal weight\\
Determines $\lambda_{\min}$ & yes & yes & yes (first birth of $\DPH_0$, when the divisors are those over $p$) & itself\\
Determines $m$ & no & yes & \textbf{no} (Rem.~\ref{rem:forgets}) & itself\\
Determines $n$ in $A_n$ & no & yes & yes (Ex.~\ref{ex:Anexplicit}) & no\\
\bottomrule
\end{tabularx}
\caption{Which property belongs to which object. In particular, sensitivity to cycles and the detection of $n$ in $A_n$ are properties of the filtered complex and of the module and not of the asymptotic pair, while the volume asymptotics and the factor $m$ belong to the pair and to the filtered complex, not to the module.}
\label{tab:hierarchy}
\end{table}

The consequences for how the results are stated are the following. The volume asymptotics, the identification of $\lambda_{\min}$ with the Mather log canonical threshold, and the factor $m$ belong to level~(4). Independence of the resolution and the barcode belong to levels~(2)--(3). No statement about volume growth is made for the later bars of~(3). Over a complete $V(\I)$ and for algebraic pairs, level~(3) with rational coefficients is the weight-zero part of a mixed Hodge structure (Proposition~\ref{prop:hodge}).

\section{Normal surface singularities: explicit persistence}\label{sec:surfaces}

\subsection{General normal surface germs}
Let $(X,0)$ be a normal complex analytic surface germ and $\I\subseteq\mathfrak m_0$ an $\mathfrak m_0$-primary ideal, so that $D_\pi=\Exc(\pi)$ for a log-resolution $\pi$, all $\pi(E)=\{0\}$. As throughout Section \ref{sec:dual}, ``log-resolution'' means Jacobian-adapted log-resolution when $X$ is singular; Example \ref{ex:A2fail} shows this to be necessary already on surface germs. Throughout this section $w_E=\lambda_E$ is the weight. The dual complex is the weighted dual graph $(\Gamma(\pi),w)$: one vertex per exceptional curve, one edge per intersection point (loops for nodes of a component, multiple edges for curves meeting several times), and $\Delta^{\le\alpha}=\Gamma(\pi)^{\le\alpha}$ is the full subgraph on the vertices with $w_E\le\alpha$. ($\Gamma(\pi)$ denotes the graph, $\Gamma_\pi(X,\I)$ the spectrum.)

\begin{lemma}\label{lem:surfdom}
Any two Jacobian-adapted log-resolutions $\pi_1\colon Y_1\to X$, $\pi_2\colon Y_2\to X$ of a normal surface germ are dominated by a common log-resolution $W\to X$ obtained from each of $Y_1,Y_2$ by a finite sequence of point blow-ups, and every model in either sequence is Jacobian-adapted; consequently $(WF)$ holds for the pair $(X,\I)$, unconditionally.
\end{lemma}
\begin{proof}
A proper birational morphism between smooth surface germs factors as a finite sequence of point blow-ups (Zariski; the statement is local and analytic). Hence there is a smooth surface $W_0$ with morphisms $W_0\to Y_1$ and $W_0\to Y_2$, each a finite sequence of point blow-ups; in dimension two every point of a model with simple normal crossings boundary has one of the three local types $(r,s)=(2,2),(2,1),(2,0)$ of Lemma \ref{lem:mediant} (a normal crossing of the boundary, a smooth point of one component, or a point off the boundary), so every point blow-up encountered is an snc blow-up in the sense of Definition \ref{def:sncblowup}, and we may take $W_0=W$ to be a log-resolution over each of $Y_1,Y_2$. Since $Y_1$ is Jacobian-adapted, Lemma \ref{lem:adaptdom}(i) applied to $W\to Y_1$ (any proper birational morphism between manifolds) gives that $W$, and every intermediate model of the sequence $Y_1=W_0'\to W_1'\to\cdots\to W$ (each again dominating $Y_1$ by the remaining blow-ups), is Jacobian-adapted; likewise for the sequence from $Y_2$. Thus the chain $Y_1\to\cdots\to W\to\cdots\to Y_2$ (blow-ups followed by inverse blow-ups) consists entirely of snc blow-ups between Jacobian-adapted log-resolutions of $(X,\I)$, which is Hypothesis \ref{hyp:WF}.
\end{proof}

\begin{theorem}\label{thm:surface}
For a normal surface germ and an $\mathfrak m_0$-primary ideal, $\DPH^\alpha_k(X,\I)=H_k(\Gamma(\pi)^{\le\alpha})$, computed on any Jacobian-adapted log-resolution $\pi$, is independent of that resolution, without any further hypothesis, and vanishes for $k\ge2$; $\DPH^\alpha_0=\Z^{b_0(\Gamma^{\le\alpha})}$ and $\DPH^\alpha_1=\Z^{b_1(\Gamma^{\le\alpha})}$ with $b_1=\#\text{edges}-\#\text{vertices}+b_0$. Under an snc blow-up of an adapted log-resolution at a point $p$: if $p=E_i\cap E_j$ the edge $v_iv_j$ is subdivided by $v_F$ with $w_F=(\nu_iw_i+\nu_jw_j)/(\nu_i+\nu_j)$; if $p$ lies on $E_i$ only, a leaf $v_F$ is attached to $v_i$, with $w_F=(a_i+2)/\nu_i>w_i$; if $p\notin\Exc$, a vertex of weight $+\infty$ appears. None of these changes $H_\bullet(\Gamma^{\le\alpha})$ for any $\alpha$; and, on a resolution that is not adapted, the first case can instead disconnect the graph before $v_F$ arrives (Example \ref{ex:A2fail}), which is exactly why only adapted resolutions are used here.
\end{theorem}
\begin{proof}
The three cases at an adapted log-resolution are Lemma \ref{lem:mediant} with $(r,s)=(2,2),(2,1),(2,0)$, where the first uses part (ii) (two-sided bound, needing adaptedness) and the other two use only part (i); and Theorem \ref{thm:blowup-inv}. That the subdivided edge does not disconnect $\Gamma^{\le\alpha}$ is precisely the upper mediant bound: $w_F$ lies between $w_i$ and $w_j$, so $v_F$ enters no later than the later of $v_i,v_j$. Independence of the adapted resolution is Theorem \ref{thm:invariance} together with Lemma \ref{lem:surfdom}, which verifies $(WF)$ directly for surface germs, without Proposition \ref{prop:WFalg}.
\end{proof}

\begin{remark}\label{rem:disconnected}
Since $\Gamma(\pi)^{\le\alpha}$ is a graph, its persistent homology has a direct reading. $\DPH^\alpha_0=H_0(\Gamma^{\le\alpha})$ counts the \emph{groups} of exceptional curves of weight $\le\alpha$ that are connected to one another through curves of weight $\le\alpha$: two curves that meet on the resolution but are separated, in the graph, by a curve of larger weight belong to different components until that curve enters. $\DPH^\alpha_1=H_1(\Gamma^{\le\alpha})$ counts the independent \emph{cycles} of the resolution graph all of whose vertices have entered. As $\alpha$ grows, components merge and cycles close; the death of a component in $\DPH_0$ at level $\beta$ means that a curve of weight exactly $\beta$ joins two groups that had been separate, and the birth of a class in $\DPH_1$ at level $\beta$ means that a curve of weight exactly $\beta$ is the last vertex of a cycle. In particular the subgraph $\Gamma^{\le\alpha}$ is not connected in general even when $\Gamma$ is a tree, so $\DPH_0$ can have rank $>1$ at intermediate levels; and $H_1(\Gamma^{\le\alpha})$ can be non-zero before all vertices have entered if $\Gamma$ has several cycles. Both phenomena occur below. Note also what the module does \emph{not} see: the self-intersection numbers, the genera of the curves, and the analytic type of the germ enter only through the weights $w_E$; two germs with the same weighted dual graph have the same module.
\end{remark}

\subsection{The rational double points $A_n$}
We now compute the module for the simplest infinite family, with the actual Jacobian orders of Remark \ref{rem:mather}, and we determine exactly when the minimal resolution is Jacobian-adapted.

\begin{example}\label{ex:An}
$X=\{xy=z^{n+1}\}\subset\C^3$, $\I=\mathfrak m$, minimal resolution $\pi_0\colon Y_0\to X$ with chain $E_1-\dots-E_n$. The valuations satisfy $v_i(x)=i$, $v_i(y)=n+1-i$, $v_i(z)=1$, so $\nu_i=1$; $k_{E_i}(X)=0$ and $j_X=(y,x,(n+1)z^n)$ give (Remark \ref{rem:mather})
\[
a_i=\ord_{E_i}(j_X)=\min(i,\ n+1-i),\qquad\lambda_i=1+\min(i,\ n+1-i)
\]
(the term $n\,v_i(z)=n$ is never the minimum). Thus $\lambda_1=\lambda_n=2$ and the weights increase towards the middle of the chain, symmetrically. Here $a_i=\ord_{E_i}(j_X)$ because the minimal resolution is crepant, $k_{E_i}(X)=0$ (Remark \ref{rem:mather}); the $a_i$ are orders of the Jacobian ideal $j_X$ of $X$, not coefficients of a relative canonical divisor, which vanishes. The computation of the $a_i$ says nothing yet about whether $\Jac_{Y_0/X}=j_X\mathcal O_{Y_0}$ is \emph{principal}, i.e.\ whether $\pi_0$ is Jacobian-adapted (Definition \ref{def:adapted}), which is a local question at each node $Z_i=E_i\cap E_{i+1}$, $i=1,\dots,n-1$.

At $Z_i$, standard toric coordinates for the $A_n$ resolution (the fan with rays $v_j=(1,j)$, $j=0,\dots,n+1$, and $x,y,z$ the monomials of weights $(0,1),(n+1,-1),(1,0)$) give local coordinates $(u,v)$ with $E_i=\{u=0\}$, $E_{i+1}=\{v=0\}$ and, up to units,
\[
x=u^iv^{i+1},\qquad y=u^{n+1-i}v^{n-i},\qquad z=uv .
\]
Hence $j_X=(x,y,z^n)=(u^iv^{i+1},\,u^{n+1-i}v^{n-i},\,u^nv^n)$, a monomial ideal, and its order at $Z_i$ (the smallest $u,v$-degree among the three generators) is
\[
\ord_{Z_i}(j_X)=\min(2i+1,\ 2n+1-2i,\ 2n).
\]
Comparing with $a_i+a_{i+1}=\min(i,n+1-i)+\min(i+1,n-i)$, a short computation (splitting into $i\le(n-1)/2$, the symmetric range $i\ge(n+1)/2$, and the one or two central values) gives: the excess $e_{Z_i}=\ord_{Z_i}(j_X)-(a_i+a_{i+1})$ is $0$ except when $n$ is even and $i=n/2$, the unique central node, where $e_{Z_{n/2}}=1$.

\begin{enumerate}[label=\textup{(\roman*)},leftmargin=2em]
\item If $n$ is odd, every excess vanishes: $\pi_0$ is already Jacobian-adapted, and the weights above are exact for the adapted resolution.
\item If $n=2m$ is even, $\pi_0$ is \emph{not} adapted: at the central node, $a_m=a_{m+1}=m$, $\nu_m=\nu_{m+1}=1$, and locally $\Jac_{Y_0/X}=u^mv^m\cdot(u,v)$ (the same pattern as Example \ref{ex:A2fail}, which is the case $m=1$), not principal. Blowing up $Z_{n/2}$ once, with new divisor $F$, gives $\nu_F=2$, $a_F=a_m+a_{m+1}+1+1=2m+2$ (Lemma \ref{lem:mediant}, with $e_{Z_m}=1$), so
\[
\lambda_F=\frac{a_F+1}{\nu_F}=\frac{2m+3}{2}=m+\tfrac32>\max(\lambda_m,\lambda_{m+1})=m+1
\]
Since $\Jac_{Y_0/X}$ near $Z_m$ is a principal monomial ideal times the maximal ideal $(u,v)$, and blowing up a point always principalizes its maximal ideal, $\Jac_{Y'/X}$ is invertible near both new nodes $E_m\cap F$ and $F\cap E_{m+1}$ (this is checked exactly as in Example \ref{ex:A2fail}); every other node is untouched. Hence one further blow-up, at the unique node where two curves of equal, maximal weight meet, suffices: the resulting resolution $\pi_1\colon Y_1\to X$, with chain $E_1-\cdots-E_m-F-E_{m+1}-\cdots-E_n$ ($n+1$ curves), is Jacobian-adapted, and it is this resolution, not $\pi_0$, that computes $\DPH$.
\end{enumerate}
\end{example}

\begin{example}\label{ex:Anexplicit}
Write the adapted resolution as $E_1-\cdots-E_n$ if $n$ is odd, and $E_1-\cdots-E_m-F-E_{m+1}-\cdots-E_n$ ($n=2m$) if $n$ is even, with weights $\lambda_i=1+\min(i,n+1-i)$ on the original curves and, when $n$ is even, $\lambda_F=m+\tfrac32$ inserted between the two curves of weight $m+1$. In either case the graph is a path (a tree), so $\DPH_1=0$, and the two ends have $\lambda_1=\lambda_n=2=\lct$. Reading the sublevel graph from both ends inward, the two arms are disjoint paths until the middle vertex enters; that vertex is $E_{(n+1)/2}$, of weight $\tfrac{n+3}2$, when $n$ is odd, and $F$, of weight $\tfrac{n+3}2=m+\tfrac32$, when $n$ is even -- the same formula in both cases. Hence
\[
\DPH_0\ \text{has barcode}\ \Bigl\{[2,\infty),\ \bigl[2,\ \tfrac{n+3}2\bigr)\Bigr\},\qquad n\ge2,
\]
with no finite bar for $n=1$ ($A_1$ is a separate case). In particular, for $n\ge2$ the length of the finite bar is $\tfrac{n-1}2$: the module determines $n$ within the $A_n$ family, unlike $\lct=2$ and $m(0)=1$, which are the same for every $n\ge3$. The table below spells this out for small $n$ ($v_F$ marked in bold where it occurs).
\begin{center}\small
\begin{tabular}{@{}llll@{}}
\toprule
 & weights (adapted) & $\Gamma^{\le\alpha}$ by level & $\DPH_0$\\
\midrule
$A_1$ & $(2)$ & $\alpha\ge2$: $v_1$ & $\Z$ from $2$, no finite bar\\[2pt]
$A_2$ & $(2,\mathbf{2.5},2)$ & $2\le\alpha<2.5$: $v_1\,|\,v_2$;\ $\alpha\ge2.5$: $v_1\mathbf{v_F}v_2$ & $\Z^2$ on $[2,2.5)$, then $\Z$\\[2pt]
$A_3$ & $(2,3,2)$ & $2\le\alpha<3$: $v_1\,|\,v_3$;\quad $\alpha\ge3$: $v_1v_2v_3$ & $\Z^2$ on $[2,3)$, then $\Z$\\[2pt]
$A_4$ & $(2,3,\mathbf{3.5},3,2)$ & $2\le\alpha<3$: $v_1\,|\,v_4$;\ $3\le\alpha<3.5$: $v_1v_2\,|\,v_3v_4$;\ $\alpha\ge3.5$: all & $\Z^2$ on $[2,3.5)$, then $\Z$\\[2pt]
$A_5$ & $(2,3,4,3,2)$ & $2\le\alpha<3$: $v_1\,|\,v_5$; $3\le\alpha<4$: $v_1v_2\,|\,v_4v_5$; $\alpha\ge4$: all & $\Z^2$ on $[2,4)$, then $\Z$\\
\bottomrule
\end{tabular}
\end{center}
For $n$ odd this is the computation on the minimal resolution, which is already adapted ($A_1,A_3,A_5$ above are unchanged). For $n$ even the extra vertex $F$ is essential: the (non-adapted) minimal resolution of $A_2$ has both curves entering at $\alpha=2$ together with their edge, hence no finite bar at all instead of the bar $[2,2.5)$, and that of $A_4$ places the merge at $\alpha=3$ instead of $3.5$. These are not the values computed by an adapted resolution and, by Example \ref{ex:A2fail}, do not compute an invariant of the pair (Figure~\ref{fig:An}).
\end{example}

\begin{figure}[ht]
\centering
\begin{tikzpicture}[x=1cm,y=1cm]
 \begin{scope}
  \node[lab] at (1.7,1.55) {$A_5$ (minimal resolution is adapted)};
  \foreach \i/\w in {0/2,1/3,2/4,3/3,4/2}{
    \node[vt,label={[lab]above:$\w$}] (p\i) at (\i*0.85,1.0) {};}
  \foreach \i/\j in {0/1,1/2,2/3,3/4}{\draw[eon] (p\i)--(p\j);}
  \draw[->,black!60] (-0.2,-0.9)--(4.0,-0.9) node[right,lab]{$\alpha$};
  \foreach \a in {2,3,4}{\draw[black!60] (\a*1.05-1.85,-0.95)--(\a*1.05-1.85,-0.85); \node[lab,below] at (\a*1.05-1.85,-0.95) {$\a$};}
  \draw[bar] (0.25,0.35)--(3.9,0.35);
  \draw[barfin] (0.25,-0.25)--(2.35,-0.25);
  \node[lab,left] at (0.2,0.35) {$[2,\infty)$}; \node[lab,left] at (0.2,-0.25) {$[2,4)$};
 \end{scope}
 \begin{scope}[xshift=7.2cm]
  \node[lab] at (1.7,1.55) {$A_4$ (after one extra blow-up at the centre)};
  \foreach \i/\w in {0/2,1/3,2/{3.5},3/3,4/2}{
    \node[vt,label={[lab]above:$\w$}] (p\i) at (\i*0.85,1.0) {};}
  \node[lab,below=1pt] at (p2) {$v_F$};
  \foreach \i/\j in {0/1,1/2,2/3,3/4}{\draw[eon] (p\i)--(p\j);}
  \draw[->,black!60] (-0.2,-0.9)--(4.0,-0.9) node[right,lab]{$\alpha$};
  \foreach \a in {2,3,3.5}{\draw[black!60] (\a*1.05-1.85,-0.95)--(\a*1.05-1.85,-0.85); \node[lab,below] at (\a*1.05-1.85,-0.95) {$\a$};}
  \draw[bar] (0.25,0.35)--(3.9,0.35);
  \draw[barfin] (0.25,-0.25)--(1.825,-0.25);
  \node[lab,left] at (0.2,0.35) {$[2,\infty)$}; \node[lab,left] at (0.2,-0.25) {$[2,3.5)$};
 \end{scope}
\end{tikzpicture}
\caption{Example~\ref{ex:Anexplicit}: weights $\lambda_i$ along the adapted resolution of $A_n$ and the barcode of $\DPH_0$. The finite bar ends when the middle vertex enters and joins the two arms; its length is $\tfrac{n-1}2$.}\label{fig:An}
\end{figure}

Two remarks on the previous table. First, $A_3$ and $A_5$ have the same barcode endpoint pattern in that both are already adapted on the minimal resolution, but their finite bars differ in length ($1$ versus $2$): the module is not constant across odd $n$. Second, on the adapted resolution the finite bar in degree $0$ has length exactly $\tfrac{n-1}2$ for every $n\ge2$ (Example \ref{ex:Anexplicit}; $A_1$ has none), so, unlike $\lct=2$ and $m(0)=1$, which coincide for all $n\ge3$, $\DPH_0$ recovers $n$ from the germ.

\begin{remark}\label{rem:wh}
For $X=\{f=0\}$ with $f$ weighted homogeneous with an isolated singularity and $\I=\mathfrak m$, the minimal good resolution graph is star-shaped, and $\Gamma^{\le\alpha}$ is disconnected until the central vertex enters. The tabulation of $\nu$ and $a$ along the arms (from the weights and continued fractions), and the verification of adaptedness this now requires, requires a verification of adaptedness. A cycle in the graph, which is what a non-rational singularity (such as a cusp) would add to this picture, is illustrated instead in Example \ref{ex:planecycle}, on a smooth $X$ where no such verification is needed.
\end{remark}

\section{Conclusion}\label{sec:outlook}

Starting from the single number $\lct(\I)$, the pair $(X,\I)$ carries a finite set of divisorial weights $\Gamma_\pi(X,\I)$, and the divisorial persistent homology $\DPH$ organizes them together with the incidence structure of the divisors. The volume asymptotics (Theorem \ref{thm:volume}) ties the first weight and the resonance multiplicity to the growth of $\Hn\{K_\I\le\varepsilon\}$, and the local threshold $\lambda_{\min}$ is intrinsic (Theorem \ref{thm:divcomp}). On Jacobian-adapted resolutions -- the class singled out by Definition \ref{def:adapted}, which Example \ref{ex:A2fail} shows to be necessary when $X$ is singular -- $\DPH$ is invariant under every snc blow-up (Theorem \ref{thm:blowup-inv}), independent of the adapted resolution unconditionally for normal surface germs (Theorem \ref{thm:surface}) and for complex algebraic pairs (Corollary \ref{cor:WFalg}), and under the factorization hypothesis \textup{(WF)} in general (Theorem \ref{thm:invariance}). The $A_n$ singularities (Example \ref{ex:Anexplicit}) show that $\DPH$ \emph{detects}, within that family, the integer $n$, which $\lct$ and the resonance multiplicity $m(0)$ alone do not; this is not a classification (Remark \ref{rem:detection}). Over a complete $V(\I)$ and for algebraic pairs, $\DPH$ with rational coefficients is the weight-zero part of the persistence mixed Hodge structure of the divisors of weight at most $\alpha$ (Proposition~\ref{prop:hodge}).

\subsection*{Open problems}
The only conditional point is \textup{(WF)} for non-algebraic complex analytic pairs of dimension $\ge3$ (Remark \ref{rem:status}, Table~\ref{tab:scope}). The main open problem of a different nature is analytic: the later critical values and classes of $\DPH$ are not interpreted. Fix $p\in V(\I)$ and a cutoff $\varphi$ as in Remark~\ref{rem:zeta}, and let $P_p$ be the set of poles of the meromorphic continuation of $\zeta_\varphi$, taken over all cutoffs supported near $p$; by Remark~\ref{rem:zeta} the set $P_p$ is contained in $\{-\lambda_E-k/\nu_E:E\subset D_\pi,\ \nu_E>0,\ k\in\Z_{\ge0}\}$, and its largest element is $-\lambda_{\min}(p)$. Three questions are natural.

\begin{question}[Critical values and poles]\label{q:poles}
Let $c$ be an effective critical value of $\DPH_k(X,\I)$ (the smallest $\alpha$ with $\DPH^\alpha\ne0$ is the case $c=\lambda_{\min}(p)$, which is settled by Remark~\ref{rem:zeta}). Is $-c\in P_p$? If the answer is negative in general, characterize the effective critical values $c$ for which it holds, in terms of the incidence of the strata of weight $c$.
\end{question}

\begin{question}[Critical values and jumping numbers]\label{q:jump}
Let $X$ be smooth and $c$ an effective critical value of $\DPH_k(X,\I)$, so that $c=\lambda_E=(a_E+1)/\nu_E$ for some divisor $E$ of weight exactly $c$; the jumping numbers of the multiplier ideals $\mathcal J(\I^{c'})$ are among the numbers $(a_E+1+k)/\nu_E$, $k\in\Z_{\ge0}$, and $c$ is the case $k=0$. Is $c$ a jumping number of $\I$? For the first critical value $c=\lct(\I)$ the answer is yes. For singular $X$ the analogous question concerns the Mather multiplier ideals \cite{dFD14,EI15} and the Mather log canonical threshold \cite{Ish13}; we have not studied it.
\end{question}

\begin{question}[Classes and resolution data]\label{q:classes}
Let $z$ be a class of $\DPH_k$ born at $c$, represented by a cycle of strata $E_\sigma$ of weight $\le c$. Is there an analytic functional of the pair, for instance the leading coefficient or the order at $s=-c$ of a localized zeta function $\zeta_\varphi$ with $\varphi$ supported near $\pi(E_\sigma)$, which is non-zero exactly when the class is born, and which depends only on the class and not on the resolution?
\end{question}

Two observations indicate what a positive answer must respect and what makes the questions non-trivial. First, a leaf attached by blowing up a point of a curve $E_i$ has weight $\lambda_i+1/\nu_i$ (Theorem~\ref{thm:surface}), which is exactly the position $k=1$ of the candidate poles of $E_i$ itself; it adds a contractible leaf and no class, consistently with it contributing no independent pole, but this is a heuristic, not a proof. Second, the residues at $s=-\lambda_E$ from different charts need not have the same sign, since a factor $\nu_j(\lambda_j-\lambda_E)$ appears for each chart divisor $E_j$ with $\lambda_j<\lambda_E$, so cancellations between charts are possible and Question~\ref{q:poles} cannot be settled by positivity alone. A first test case is $A_n$ with $\I=\mathfrak m$: the finite bar of $\DPH_0$ ends at $\lambda=(n+3)/2$ (Example~\ref{ex:Anexplicit}), and one can ask whether $\zeta_\varphi$ has a pole at $s=-(n+3)/2$; we have not computed this. Nothing in the present paper bears on these questions beyond Remark~\ref{rem:zeta}.

The weights entering $\DPH$ arise from the monomial change-of-variables formula of Section \ref{sec:volume}, and their local minimum determines the leading volume exponent (Theorem \ref{thm:volume}); the other weights are not claimed to be exponents of an expansion of the volume. For complex pairs, domination (Proposition \ref{prop:domination}) uses Hironaka's flattening theorem in the analytic case and resolution of indeterminacy in the algebraic case; weak factorization is available for algebraic pairs (Proposition \ref{prop:WFalg}) and not claimed for non-algebraic germs of dimension $\ge3$. The real analytic case, in which the weights are halved and domination needs complexification and descent, is treated in \cite{GruR}.


\begin{thebibliography}{99}

\bibitem{AKMW02} D. Abramovich, K. Karu, K. Matsuki, and J. W\l odarczyk, Torification and factorization of birational maps, J. Amer. Math. Soc. 15 (2002), no. 3, 531--572.

\bibitem{BM88} E. Bierstone and P. D. Milman, Semianalytic and subanalytic sets, Publ. Math. Inst. Hautes \'Etudes Sci. 67 (1988), 5--42.

\bibitem{BM97} E. Bierstone and P. D. Milman, Canonical desingularization in characteristic zero by blowing up the maximum strata of a local invariant, Invent. Math. 128 (1997), 207--302.

\bibitem{BCR98} J. Bochnak, M. Coste, and M.-F. Roy, Real Algebraic Geometry, Ergeb. Math. Grenzgeb. (3) 36, Springer, Berlin, 1998.

\bibitem{BFJ08} S. Boucksom, C. Favre, and M. Jonsson, Valuations and plurisubharmonic singularities, Publ. Res. Inst. Math. Sci. 44 (2008), no. 2, 449--494.

\bibitem{vdD98} L. van den Dries, Tame topology and o-minimal structures, London Math. Soc. Lecture Note Ser., vol. 248, Cambridge Univ. Press, 1998.

\bibitem{EI15} L. Ein and S. Ishii, Singularities with respect to Mather--Jacobian discrepancies, in: Commutative Algebra and Noncommutative Algebraic Geometry I, MSRI Publ. 67, Cambridge Univ. Press, 2015, pp. 125--168.

\bibitem{EIM16} L. Ein, S. Ishii, and M. Musta\c t\u a, Multiplier ideals via Mather discrepancy, in: Minimal Models and Extremal Rays (Kyoto, 2011), Adv. Stud. Pure Math. 70, Math. Soc. Japan, Tokyo, 2016, pp. 9--28.

\bibitem{Del71} P. Deligne, Th\'eorie de Hodge. II, Publ. Math. Inst. Hautes \'Etudes Sci. 40 (1971), 5--57.

\bibitem{Del74} P. Deligne, Th\'eorie de Hodge. III, Publ. Math. Inst. Hautes \'Etudes Sci. 44 (1974), 5--77.

\bibitem{dFD14} T. de Fernex and R. Docampo, Jacobian discrepancies and rational singularities, J. Eur. Math. Soc. (JEMS) 16 (2014), no. 1, 165--199.

\bibitem{dFKX17} T. de Fernex, J. Koll\'ar, and C. Xu, The dual complex of singularities, in: Higher dimensional algebraic geometry, Adv. Stud. Pure Math. 74 (2017), 103--129.

\bibitem{GM88} M. Goresky and R. MacPherson, Stratified Morse Theory, Ergeb. Math. Grenzgeb. (3) 14, Springer, Berlin, 1988.

\bibitem{GruR} N. Grulha, Divisorial persistent homology of real analytic pairs, in preparation.

\bibitem{Gru26} N. G. Grulha Jr., On the divisorial geometry of volume asymptotics of sublevel sets, arXiv:2606.30171 [math.AG] (2026).


\bibitem{Har75} R. M. Hardt, Stratification of real analytic mappings and images, Invent. Math. 28 (1975), 193--208.

\bibitem{Hir64} H. Hironaka, Resolution of singularities of an algebraic variety over a field of characteristic zero, Ann. of Math. (2) 79 (1964), 109--326.

\bibitem{Hir75} H. Hironaka, Flattening theorem in complex-analytic geometry, Amer. J. Math. 97 (1975), no. 2, 503--547.

\bibitem{Ish13} S. Ishii, Mather discrepancy and the arc spaces, Ann. Inst. Fourier (Grenoble) 63 (2013), no. 1, 89--111.

\bibitem{JM12} M. Jonsson and M. Musta\c t\u a, Valuations and asymptotic invariants for sequences of ideals, Ann. Inst. Fourier (Grenoble) 62 (2012), no. 6, 2145--2209.

\bibitem{Kol97} J. Koll\'ar, Singularities of pairs, Proc. Sympos. Pure Math. 62, Part 1, Amer. Math. Soc., 1997, pp. 221--287.

\bibitem{Kol07} J. Koll\'ar, Lectures on Resolution of Singularities, Annals of Math. Studies, vol. 166, Princeton Univ. Press, Princeton, NJ, 2007.

\bibitem{KX16} J. Koll\'ar and C. Xu, The dual complex of Calabi--Yau pairs, Invent. Math. 205 (2016), no. 3, 527--557.

\bibitem{KW26} D. Kosta and D. Windisch, Classification of real hyperplane singularities by real log canonical thresholds, SIAM J. Appl. Algebra Geom. 10 (2026), no. 2, 238--260.

\bibitem{Li18} C. Li, Minimizing normalized volumes of valuations, Math. Z. 289 (2018), no. 1--2, 491--513.

\bibitem{LLX20} C. Li, Y. Liu, and C. Xu, A guided tour to normalized volume, in: Geometric analysis -- in honor of Gang Tian's 60th birthday, Progr. Math. 333, Birkh\"auser/Springer, 2020, pp. 167--219.

\bibitem{Mus12} M. Musta\c t\u a, IMPANGA lecture notes on log canonical thresholds, in Contributions to Algebraic Geometry, Eur. Math. Soc., Z\"urich, 2012, pp. 407--442.

\bibitem{MN15} M. Musta\c t\u a and J. Nicaise, Weight functions on non-archimedean analytic spaces and the Kontsevich--Soibelman skeleton, Algebr. Geom. 2 (2015), no. 3, 365--404.

\bibitem{NX16} J. Nicaise and C. Xu, Poles of maximal order of motivic zeta functions, Duke Math. J. 165 (2016), no. 2, 217--243.

\bibitem{Nob75} A. Nobile, Some properties of the Nash blowing-up, Pacific J. Math. 60 (1975), no. 1, 297--305.

\bibitem{Pay13} S. Payne, Boundary complexes and weight filtrations, Michigan Math. J. 62 (2013), no. 2, 293--322.

\bibitem{Sai07} M. Saito, On real log canonical thresholds, preprint, 2007, arXiv:0707.2308.

\bibitem{Ste06} D. A. Stepanov, A remark on the dual complex of a resolution of singularities, Uspekhi Mat. Nauk 61 (2006), 185--186.

\bibitem{Wat09} S. Watanabe, Algebraic geometry and statistical learning theory, Cambridge Monogr. Appl. Comput. Math., vol. 25, Cambridge Univ. Press, 2009.

\bibitem{Wat24} S. Watanabe, Recent advances in algebraic geometry and Bayesian statistics, Inf. Geom. 7 (Suppl. 1), S187--S209 (2024).

\bibitem{Wlo03} J. W\l odarczyk, Toroidal varieties and the weak factorization theorem, Invent. Math. 154 (2003), no. 2, 223--331.

\bibitem{XZ21} C. Xu and Z. Zhuang, Uniqueness of the minimizer of the normalized volume function, Camb. J. Math. 9 (2021), no. 1, 149--176.

\end{thebibliography}
\end{document}